\documentclass[12pt, reqno]{article}
\usepackage{fullpage}
\usepackage{amssymb, amsmath, amsthm, amsrefs}
\usepackage{mathrsfs}
\usepackage{enumerate}
\usepackage[colorlinks,breaklinks, allcolors=blue]{hyperref}
\usepackage{xcolor}
\usepackage{longtable}
\usepackage{cleveref}
\usepackage{tikz-cd}
\usepackage{caption}

\newtheorem{theorem}{Theorem}[section]
\newtheorem{lemma}[theorem]{Lemma}

\newtheorem{corollary}[theorem]{Corollary}

\theoremstyle{definition}
\newtheorem{definition}[theorem]{Definition}
\newtheorem{remark}[theorem]{Remark}

\numberwithin{equation}{section}

\newcommand{\U}{\mathrm{U}}

\newcommand{\Hom}{\mathrm{Hom}}
\newcommand{\SO}{\mathrm{SO}}

\newcommand{\rk}{\mathrm{rk}}

\DeclareMathOperator{\Aut}{Aut}

\theoremstyle{definition}

\newcommand{\C}{\mathbb C}
\newcommand{\Z}{\mathbb Z}

\newcommand{\Q}{\mathbb Q}
\newcommand{\R}{\mathbb R}

\newcommand{\wh}{\; | \;}

\newcommand{\eps}{\varepsilon}

\begin{document}
	\title{Characterization of locally standard, $\Z$-equivariantly formal manifolds in general position}
	\author{Nikolas Wardenski \footnote{University of Haifa, email: n.wardenski@web.de}}
	\maketitle
 \begin{abstract}
		We give a characterization of locally standard, $\Z$-equivariantly formal manifolds in general position. In particular, we show that for dimension $2n$ at least $10$, to every such manifold
		with labeled GKM graph $\Gamma$ there is an equivariantly formal torus manifold such that the restriction of the $T^n$-action to a certain $T^{n-1}$-action yields the
		same labeled graph $\Gamma$, thus showing that the (equivariant) cohomology with $\Z$-coefficients of those manifolds has the same description as that of equivariantly
		formal torus manifolds.
	\end{abstract}
	
\section{Introduction}
In \cite{MP03}, Masuda and Panov studied the (equivariant) cohomology of locally standard torus manifolds, especially that of equivariantly formal ones.
Further, it was shown that a torus manifold is equivariantly formal over $\Z$ if and only if it is locally standard and its orbit space is
face-acyclic, that is, the orbit space itself and the orbit space of all face submanifolds are acyclic.\\
Building on this, Ayzenberg and Masuda (\cite{AM23}) studied $R$-equivariantly formal $T$-actions in general position on compact manifolds (where $R=\Z$ if
all stabilizers are connected or $R=\Q$ in general) in a similar fashion. They showed that the orbit space of any such manifold has the $R$-homology of a sphere
and also found a partial converse to this statement, which deduces the equivariant formality over $R$ of a $T$-manifold in general position from
the assumption that its orbit space has the $R$-homology of a sphere and that its faces are acyclic. However, for this converse it is needed that all stabilizers are connected,
even for $R=\Q$.\\

Our main goal here is to extend these results for $R=\Z$ to locally standard $T$-actions in general position.
At first, we have a sufficient condition for equivariant formality.
\begin{theorem}
	Consider an action in general position of $T=T^{n-1}$ on the compact manifold $M$ of dimension $2n\geq 8$ such that
	\begin{itemize}
		\item for all closed subgroups $H\subset T$, every connected component of $M^H$ intersects $M^T$ non-trivially.
		\item the orbit space $M/T$ has the homology of a sphere (and is thus a homology sphere).
		\item $M_{n-2}^*$ is $n-3$-acyclic over $\Z$.
		\item every face-submanifold of codimension $4$ is equivariantly formal.
		\item for any isotropy submanifold $Q$ fixed by $\Z_p\subset T$, corresponding to an ineffective subgraph of covalence $1$,
		the map $H_*(Q^*)\to H_*(Q^*)$ given by multiplication with $p$ is an isomorphism in positive degrees.
	\end{itemize}
	Then the action is equivariantly formal over $\Z$.
\end{theorem}
If an isotropy submanifold $Q$ as above does satisfy the assumption in the fifth point, we call
$Q$ (or $Q^*$) \textit{admissible}.\\

This sufficient condition is indeed necessary.
\begin{theorem}
	Consider an action in general position of $T=T^{n-1}$ on the compact manifold $M$ of dimension $6$ or $2n\geq 10$ such that
	\begin{itemize}
		\item The odd (co)homology of $M$ vanishes.
		\item Whenever $x$ is a point whose stabilizer is finite, there is a neighborhood of $x$ in which this is the only non-trivial stabilizer subgroup.
	\end{itemize}
	The following statements hold.
	\begin{enumerate}
		\item For all closed subgroups $H\subset T$ (not necessarily connected), every connected component of $M^H$ intersects $M^T$ non-trivially.
		\item Any connected component $Q^*$ of $Z^*$ is admissible. 
		\item $M^*_k$ is $k-1$-acyclic for $k\leq n-2$.
		\item $M^*$ is a homology sphere.
	\end{enumerate}
\end{theorem}
The second assumption, together with the first one, actually implies that the action is locally standard.\\

 These results (see \cref{theosuffEF} for the sufficient and \cref{theonecEF} for the necessary condition)
 are proven in their respective sections \cref{suffEF} and \cref{necEF}.
 In the sections before that, we build some machinery to then prove these statements. For potential later use,
we work with $R$-coefficients (where $R=\Q$ or $R=\Z$) immediately, although the main results of \cref{EFgenPos} are formulated for $R=\Z$ only.\\

After that, \cref{secorient} serves as a preparation for the subsequent sections. There, we establish 'our' notion of 'orientability of a GKM graph'
(there are some others, e.g. \cite{BP15}[Definition 7.9.16] and \cite[Definition 2.19]{GKZ22}),
while constructing a smooth $T^k$-manifold with boundary $M_1'(\Gamma)$ out of any GKM graph $\Gamma$ with (signed or unsigned) $\Z^k$-labeling. We show that $M_1'(\Gamma)$ admits an equivariant deformation
retract onto its equivariant one-skeleton, and that any GKM manifold with graph $\Gamma$ is locally equivariantly homeomorphic to $M_1'(\Gamma)$.\\

Then we switch to manifolds in general position. In \cref{local} and $R=\Z$, we are working in the following setting: $M$ is a compact, connected $T$-manifold
of dimension $2n$ in general position, the action is locally standard, the faces are equivariantly formal and the orbit space of the equivariant $n-2$-skeleton, $M_{n-2}^*$, is $n-3$-acyclic.
In \cref{graph}, we show that GKM graphs with certain natural properties, which are satisfied in the above setting, come from torus graphs. That is, we show that the
$\Z^{n-1}$-labeling (signed or unsigned) of these graphs extends to a $\Z^n$-labeling (signed or unsigned).\\

Building on this, we then show in \cref{local} that the $T$-action in a neighborhood of $M_k$, $k\leq n-3$, extends naturally to a $T^n$-action,
and that, for $k\leq n-2$, there is an equivariant deformation retract $M_k'\to M_k$, where $M_k'$ is a smoothly embedded $T^n$-manifold (or $T$-manifold if $k=n-2$)
with boundary $X_k$. We close this section with a description of the homology of the topological manifold $X^*=X_{n-2}^*=X_{n-2}/T$.\\

In \cref{EFgenPos}, we study first the homological properties of $X=X_{n-2}$ and the map $H_*(X)\to H_*(M_{n-2})$. An important tool for that matter is the result, also shown in that section, that $M_{n-3}'$ admits a natural $T^n$-equivariant embedding into an equivariantly formal torus manifold $\tilde{M}$, which implies the following theorem.
\begin{theorem}
    The equivariant cohomology of a $\Z$-equivariantly formal manifold in general position is that of an equivariantly formal torus
manifold, after restricting its $T^n$-action to a certain subtorus of codimension $1$.
\end{theorem}
Then finally \cref{theosuffEF} and \cref{theonecEF} are shown in \cref{suffEF} respectively \cref{necEF}.

\newpage

\section{Preliminaries}
\subsection{GKM actions and GKM graphs}
Let $M$ be a compact, oriented smooth manifold of dimension $2n$ on which a torus $T$ of dimension $m$ acts effectively. We use the standard notations $Tx$ and $T_x\subset T$ for the orbit of $T$
through $x\in M$ and the stabilizer of $x$ in $T$, respectively. For $H$ an arbitrary subgroup of $T$, we define $M^{(H)}$ to be all elements $x\in M$ with $T_x=H$ and
$M^H$ to be all elements $x\in M$ satisfying $H\subset T_x$. We denote by $M_k$ the equivariant $k$-skeleton of $M$, that is,
\[
M_k:=\{x\in M \wh \dim(Tx)\leq k \}=\{x\in M \wh \dim(T_x)\geq m-k \}.
\]
We get a filtration of $M$ by
\[
M^T=M_0 \subset M_1\subset \hdots \subset M_{m-1}=M
\]
and an induced filtration of the orbit space $M/T=M^*$, where for any $T$-invariant set $X\subset M$ we set $X^*$ to be the image of the projection $\pi\colon M\to M/T$.
From now on, we always assume that $M^T$ is finite and nonempty.
\begin{definition}
	The closure $F$ of a connected component of $M^*_i\setminus M^*_{i-1}$ in $M^*$ is called a \textit{face} if it intersects $M_0^*$ non-trivially.
	We define its rank $\rk(F)$ to be the number $i$ and call $\pi^{-1}(F)$ a \textit{face submanifold}.
\end{definition}
We note that the latter definition is justified since $\pi^{-1}(F)$ is indeed a submanifold of $M$. Also, for any face $F$, we set
\[
F_{-1}:=\{x\in F\colon \dim Tx < \text{rk}(F)\}.
\]
Now we turn to GKM actions.
For a graph $\Gamma$, we denote by $E(\Gamma)$ the set of all edges and by $V(\Gamma)$ the set of all vertices.
We can give each edge $e$ two possible orientations, each determining an initial vertex $i(e)$ and a terminal vertex $t(e)$, respectively.
On an oriented edge $e$, we denote by $\bar{e}$ the same unoriented edge with the other orientation. Thus $i(e)=t(\bar{e})$ and $t(e)=i(\bar{e})$.
We let $\widetilde{E}(\Gamma)$ be the set of all oriented edges, and let $E_v$ and $\widetilde{E}_v$ be the corresponding edges on any $v\in V(\Gamma)$.\\
An abstract unsigned GKM graph $(\Gamma,\alpha)$ consists of an $n$-valent graph $\Gamma$ (multiple edges may appear between two vertices)
and a \textit{labeling} $\alpha \colon E(\Gamma)\to \Z^{k}/ \pm 1$ such that for all $v\in V$ and any two $e_1,e_2$ at $v$, $\alpha(e_1)$ and $\alpha(e_2)$
are linearly independent, and such that there is a compatible connection $\nabla$, which we define now.
\begin{definition}
	A compatible connection $\nabla$ on $(\Gamma,\alpha)$ is a bijection $\nabla_e\colon \widetilde{E}_{i(e)}\to \widetilde{E}_{t(e)}$ (for each oriented edge $e$) such that
	\begin{enumerate}
		\item $\nabla_e e=\bar{e}$.
		\item $\nabla_{\bar{e}}=(\nabla_e)^{-1}$.
		\item for all $f \in \widetilde{E}_{i(e)}$ we have $\alpha(\nabla_e f)\pm \alpha(f)=c \alpha(e)$ for some $c\in \Z$.
	\end{enumerate}
\end{definition}
Analogously, we call $(\Gamma,\alpha)$ an abstract \textbf{signed} GKM graph when $\alpha$ takes values in $\Z^k$, $\alpha(e)=-\alpha(\bar{e})$, and there is a compatible connection
such that the third condition above is replaced with $\alpha(\nabla_e f)- \alpha(f)=c \alpha(e)$.\\
There is a rather obvious notion of isomorphism of GKM graphs.
\begin{definition}
	Two (signed) GKM graphs $(\Gamma,\alpha)$ and $(\Gamma',\alpha')$ are isomorphic if there exists an isomorphism $\Psi\colon \Gamma\to \Gamma'$ of unlabeled graphs,
	as well as an automorphism $\varphi\colon \mathfrak{t}^*\to \mathfrak{t}^*$ such that $\varphi(\alpha(e))=\alpha'(\Psi(e))$.
\end{definition}
There is a relationship between $T^k$-actions on $\C^n$ and the values of $\alpha$ on $E_v$. Namely, for
every edge $e\in E_v$, $\alpha(e)$ corresponds to a a homomorphism $\chi_e\colon T^k\to S^1$, which is well-defined up to complex conjugation on $S^1$.
In particular, the representation $\chi_v$ of $T^k$ on $\C\cong \R^2$ defined by $\chi_e$ is well defined up to (real) isomorphism, and thus we get a representation of $T^k$ on
$\C^n$ (up to real isomorphism) by
\[
t\cdot (z_1,\hdots,z_n)=(\chi_{e_1}(t) z_1,\hdots, \chi_{e_n}(t)z_n).
\]
We denote this representation by $\C^n(v)$.
\begin{definition}
    We call a GKM graph $\Gamma$ \textit{ineffective} if the representation $\C^n(v)$ is unfaithful for some vertex (and hence
    any, if the graph is connected) $v$.
\end{definition}
Note that the equivariant one skeleton $(\C^n(v))_1$ of $\C^n(v)$ is precisely the union of the single $\C$-summands,
because we assumed that any two labels are linearly independent. Of course, every $T^k$-representation on $\C^n$ such that $\C^n_1$ is the union of single $\C$-summands
is of the above form, so determines a 'labeling'. More generally, we have the following lemma.
\begin{lemma}
	Let $M$ be a possibly open manifold acted on by $T=T^k$ smoothly with the following properties:
	\begin{itemize}
		\item The set $M^T$ is finite and not empty.
		\item The equivariant one skeleton $M_1$ is given by a union of $T$-invariant 2-spheres.
	\end{itemize}
	Then the set $\Gamma=M_1/T$ has a natural graph structure (vertices correspond to fixed points, edges correspond to $T$-invariant 2-spheres), and there is a labeling
	$\alpha$ determined by the isotropy representation at each vertex. Moreover, the tuple $(\Gamma,\alpha)$ is a GKM graph.
\end{lemma}
From now on, we will omit the labeling $\alpha$ and will only write about 'the GKM graph $\Gamma$'.
\subsection{\texorpdfstring{$j$}{dsds}-independence and the formality package}
From now on, the coefficient ring $R$ for all (co)homology is taken to be either $\Q$ or $\Z$. Again, we assume $T$ acts smoothly on $M=M^{2n}$
with $M^T$ finite and not empty.
\begin{definition}
	The action of $T$ is called $j$-independent, $j\geq 2$, if for any $x\in M^T$ any $j$ weights of the tangent representation of $T$ on $T_x M$
	are linearly independent over $\Q$.\\
	If $j=n-1$ and $T=(S^1)^j=T^j$, then we say that the action is in \textit{general position}.\\
	If $j=n$ (that is, the torus has maximal dimension), then we speak of a \textit{torus manifold}.
\end{definition}
Now we define what an action of \textit{GKM$_j$-type} is. This is closely related, but not identical to the action being $j$-independent.
\begin{definition}
	A $j$-independent action is said to be of \textit{GKM$_j$-type} or \textit{GKM$_j$} if the odd cohomology of $M$ vanishes. For $j=2$, we just omit
	the index and speak of a GKM action.
\end{definition}
\begin{remark}
	At first glance, it seems weird to ask for topological properties of the manifold acted on. However, there is a well-known result that,
	when the set of fixed points is isolated, this topological restriction is equivalent to the action being equivariantly formal
	(with respect to the coefficient ring $R$ chosen). That is, denoting by $ET\to BT$ the classifying bundle of $T$, the equivariant cohomology
	\[
	H_T^*(M):=H^*(M\times_T ET)
	\]
	is a free $H^*(BT)$-module, where the module structure comes from the homomorphism $H^*(BT)\to H^*(M\times_T ET)$
	induced by the projection $M\times_T ET\to BT$. More precisely, we have the isomorphism of $H^*(BT)$-modules
	\[
	H_T^*(M)\cong H^*(BT)\otimes H^*(M),
	\]
	since the Serre spectral sequence associated to $M\to M\times_T ET\to BT$ collapses at the second page due to degree reasons
	($H^*(BT)$ is the polynomial ring in $\dim(T)$ generators of even degree). In particular, the restriction map
	$H_T^*(M)\to H^*(M)$ is surjective and its kernel is $H^{\geq 1}(BT)\otimes H^*(M)$.
\end{remark}
There is a natural way to associate a ring to a GKM graph $\Gamma$. In order to do this, we set $T$ to be a $k$-dimensional torus and note that, abstractly,
the group of homomorphisms from $T$ to $S^1$ is isomorphic to $H^2(BT;\Z)$, because both are isomorphic to $\Z^k$. However, there is even a natural isomorphism,
coming from a fact that every homomorphism $T\to S^1$ gives a map $BT\to BS^1$ and thus, after fixing some generator of $H^2(BS^1)=\Z$, a unique element in $H^2(BT)$.
So, if $\Gamma$ is signed, we can uniquely identify any weight with an element in $H^2(BT;\Z)$, and when $\Gamma$ is not signed, this only works up to a sign in $H^2(BT;\Z)$.
\begin{definition}
	Let $R$ be either $\Q$ or $\Z$.
	The equivariant cohomology $H_T^*(\Gamma;R)$ of a GKM graph $\Gamma$ is defined by
	\[
	\left\{(\omega(v))_v\in \bigoplus_{v\in V(\Gamma)}  H^*(BT;R)\colon \omega(u)-\omega(w)\equiv 0\mod \alpha(e) \text{ for all edges $e$ between $u$ and $w$}\right\},
	\]
	where $\omega(u)-\omega(w)\equiv 0\mod \alpha(e)$ means that $\omega(u)-\omega(w)$ is contained in $H^*(BT;R)\cdot \alpha(e)$.\\
	The cohomology $H^*(\Gamma;R)$ of a GKM graph $\Gamma$ is defined by
	\[
	H^*(\Gamma;R)=H_T^*(\Gamma;R)/(H^{\geq 1}(BT;R)\cdot H_T^*(\Gamma;R)).
	\]
\end{definition}
As one can already guess from the definition, the equivariant cohomology respectively the cohomology of a GKM manifold is strongly linked
to the equivariant cohomology respectively the cohomology of the graph. The next theorem is \cite[Theorem 7.2]{GKM98},
and is obtained by using the equivariant Mayer Vietoris sequence as well as the Chang-Skjelbred lemma \cite[Lemma 2.3]{CS74}.
\begin{theorem}
	Let $M$ be a GKM $T$-manifold over $\Q$. There is an isomorphism of $H^*(BT;\Q)$-algebras $H_T^*(M;\Q)\to H_T^*(\Gamma;\Q)$,
	induced by the restriction map $H_T^*(M;\Q)\to H_T^*(M^T;\Q)$. This also induces an isomorphism
	$$H_T^*(M;\Q)/(H^{\geq 1}(BT;\Q)\cdot H_T^*(M;\Q))=H^*(M;\Q)\to H^*(\Gamma;\Q).$$
\end{theorem}
\begin{remark}
    This also holds for $\Z$-coefficients if every finite stabilizer is contained in a proper subtorus, for then the statement of the Chang-Skjelbred lemma also holds for $\Z$-coefficients
    (see \cite[Theorem 2.1]{FP07}).
\end{remark}
\begin{lemma}[{\cite[Lemma 2.2]{MP03}}]
	If the action is equivariantly formal over $\Q$ or $\Z$, then, for any subtorus $H\subset T$, any connected component of $M^H\subset M$ contains a fixed point and its odd $\Q$- or $\Z$-cohomology vanishes.
\end{lemma}
When we assume 'more' linear independence together with equivariant formality, we get even stronger results.
\begin{theorem}[{\cite[Proposition 3.11]{AMS22}}]\label{orbitSpace}
	If the action is of GKM$_j$-type and $R=\Q$, then
	\begin{enumerate}
		\item for any face $F$ we have $H^i(F,F_{-1})=0$ for $i< \rk(F)$. If, in addition, $\rk(F)<j$, then $H^*(F,F_{-1})=H^*(D^{\rk(F)},\partial D^{\rk(F)})$.
		\item $M^*$ is $j+1$-acyclic (that is, $H^i(M^*)=0$ for $1\leq i\leq j+1$).
		\item $M^*_r$ is $\min(r-1,j+1)$-acyclic.
	\end{enumerate}
	If one assumes that all stabilizers are connected, then this also holds for $R=\Z$.
\end{theorem}
There is a kind of an inverse to the last theorem.
\begin{theorem}[{\cite[Theorem 4, Chapter 6]{AM23}}]
	Assume that the action of $T$ on $M$ is in general position and satisfies the following properties (here, rational or integer coefficients are taken):
	\begin{itemize}
		\item every face submanifold contains a fixed point;
		\item all stabilizers are connected;
		\item the orbit space is a homology $(n+1)$-sphere;
		\item each face of $Q_{n-2}$ is a homology disc;
		\item $Q_{n-2}$ is $(n-3)$-acyclic.
	\end{itemize}
	Then the action is equivariantly formal.
\end{theorem}

The assumption on the orbit space being a homology-sphere is not as restrictive as it seems. We call an action in general position
\textbf{appropriate}, if, for any closed subgroup $H\subset T$, the closure of $M^{(H)}$ contains a point $x'$ whose
stabilizer has a larger dimension than $H$.
\begin{lemma}\cite[Theorem 2.10]{A18}\label{lem:approp}
	The orbit space of an appropriate $T$-manifold in general position is a topological manifold.
\end{lemma}
\begin{remark}\label{rem:isotropygenpos}
	From the last lemma, one can deduce that, for any generic point $x$ in a $2k$-dimensional face submanifold $F$ of a manifold in general position, the isotropy	representation of the connected component $T_x^0$ containing the identity element of $T_x$ on the normal bundle of $F$ is in general position.
	For if not, then the orbit space would not be a topological manifold. The reason for that, in turn, is that the orbit space of the $T_x^0$-action on $S^{2n-2k-1}\subset \C^{n-k}$ is not a manifold, since when the first $n-k-1$ weights, for example, are not linearly independent, then
    the kernel of the $T_x^0$-action on $\C^{n-k-1}\times \{0\}$, which is a circle,
    rotates in the last factor of $\C$, and so the orbit space near a generic point in $\C^{n-k-1}$ has boundary.
\end{remark}
We will treat actions in general position that are in some way \textit{locally standard}, that is, every slice looks just like a slice from some linear $T$-action
on $\C^n$ (these are clearly appropriate).
The following lemma is 'standard' for a locally standard complexity one space in general position.
\begin{lemma}\label{lem:isotropycomplone}
	Whenever an action of $T=T^{n-1}$ on a $2n$-dimensional manifold is locally standard
 and in general position, a non-trivial stabilizer
	subgroup $T_x$ of any point $x$ in $M$ is either connected or a product of a subtorus with one
    non-trivial cyclic group.
\end{lemma}
\begin{proof}
	For $x\in M$ there is a representation of $T$ on $\C^n$
	together with a point $p\in \C^n$ such that $T_p=T_x$. Assume that precisely the last $n-k$ coordinates of $p$
	are $0$, so that $\dim(T\cdot p)\leq k$. On the other hand, $T_p$ acts effectively on
 $\{0\}\times \C^{n-k}\subset \C^k\times \C^{n-k}$,
	so $\dim(T_p)\leq n-k$. Using $\dim(T_p)+\dim(T\cdot p)=n-1$ we deduce $\dim(T\cdot p)\geq k-1$, that is,
	there are only the cases $\dim(T\cdot p)= k-1$ and $\dim(T\cdot p) =k$.\\
	In the former casewe have $\dim(T_p)=n-k$, so $T_p\cong T^{n-k}$ (because $T_p$ acts 
 effectively on $\C^{n-k}$). But his can not happen, since
 the orbit space of that action on $\{0\}\times \C^{n-k}$ is not a manifold, a contradiction
 to our assumption that the action of $T$ on $M$ is in general position.\\
 
 So necessarily $\dim(T\cdot p) =k$, and thus $T_p$ is a codimension one subgroup of
	$T^{n-k}$, and thus isomorphic to a product of a torus and a single cyclic group.
\end{proof}
All the next lemmata, up until the end of the subsection, are highly technical in nature, and only important
for \cref{fol} in the setting $R=\Q$.
\begin{lemma}\label{lem:normalform}
    Let $\rho$ be an unfaithful $T^n$-representation on $\C^n$, 
    where the kernel is cyclic of order $m$
    and contained in the last $S^1$-factor. Then, up to automorphism of $T^n$, the representation
    is given by $(t_1,\hdots,t_n)\cdot (z_1,\hdots,z_n)=(t_1 z_1,\hdots, t_{n-1}z_{n-1},t_n^m z_n)$.
\end{lemma}
\begin{proof}
    By dividing out the kernel, we see that the representation $\rho$ is given by
    the concatenation of the covering $(t_1,\hdots,t_n)\mapsto (t_1,\hdots,t_{n-1},t_n^m)$
    and a standard, faithful representation of $T^n$ on $\C^n$. The latter is, up
    to automorphism, given by
    $(t_1,\hdots,t_n)\cdot (z_1,\hdots,z_n)=(t_1 z_1,\hdots,t_n z_n)$.
    Now this automorphism lifts to $T^n$ under the covering, and we are done.
\end{proof}
\begin{lemma}\label{lem:isotropycomplone1}
	Let $T=T^{n-1}$ act on $\C^n$ linearly and in general position. Let $T_p=\Z_m\times T^{n-k-1}$ ($m$ could also be $1$) be an $n-k-1$-dimensional stabilizer subgroup of
	a point $p$ in $\C^{k}\times \{0\}$. Then, up to automorphism of $T$, the $T$-action on $\C^k\times \{0\}$ is given by
	$(t_1,\hdots,t_{n-1})\cdot(z_1,\hdots,z_k)=(t_1z_1,\hdots,t_{k-1}z_{k-1},t_k^m z_k)$.
\end{lemma}
\begin{proof}
	Fix the $T_p$-representation on $\C^{n-k}$. A local model around $p$ is $(-1,1)^k \times T\times_{T_p} \C^{n-k}$. Now $T_p$ is contained in a subtorus $T'$
	of dimension $n-k$, and this subtorus consists of the last $n-k$ factors of $S^1$ in $T$, up to automorphism. In the same vein, we may assume that
	the embedding $T_p\hookrightarrow T'=T^{n-k}\hookrightarrow T^{n-1}$ is of the form
	\[
	(e^{2k\pi i/m},t_1,\hdots,t_{n-k-1})\mapsto (1,\hdots,1, e^{2k\pi i/m}, t_1,\hdots,t_{n-k-1}).
	\]
	Now $T^k\times \{0\}$ acts on $\C^k\times \{0\}$ and the kernel of this action is a $\Z_m$ contained in the last $S^1$-factor, implying the assertion by means of
 \cref{lem:normalform}
\end{proof}
Starting from there, we immediately obtain the following.
\begin{lemma}
	In the setting of \cref{lem:isotropycomplone1} (and the canonical form of the $T$-action on $\C^k\times \{0\}$), we find an automorphism of $T$ that fixes the first $k-1$
 coordinates and sends $\{e\}^{k-1}\times T^{n-k}$ to itself,
 such that the resulting $T$-action on $\{0\}\times \C^{n-k}$ is given by
 \[
(t_1,\hdots,t_{n-1})\cdot (z_{k+1},\hdots, z_n)=
(t_k^lz_{k+1}, t_{k+1}z_{k+2},\hdots, t_{n-1} z_n),
 \]
 where $l$ is an integer coprime to $m$.
\end{lemma}
\begin{proof}
    Simply consider the representation of
    the last $n-k$ coordinates of $T$ on the last $n-k$
    coordinates of $\C^n$, and note that the kernel of this action is cyclic and contained in the $k$-th coordinate of $T$. Now apply \cref{lem:normalform}, again.
\end{proof}
\begin{remark}
    Note that, in the setting of the last lemma, we have no control over the action of $T$ on the $k$-th coordinate
    of $\C^n$ anymore, except that the $k$-th circle still acts with weight $m$. In that sense, we have achieved that
    the action is standard in all coordinates, except
    the $k$-th one.
\end{remark}
At last, we obtain
\begin{lemma}\label{lem:extcomplone}
	In the setting of the last lemma (and the canonical form of the $T$-action on $\C^n$), let $S^1$ act on $\C^n$ such that the $T\times S^1$-action is effective.
	Then, up to automorphism of $T\times S^1$ fixing $T\times \{e\}$, the $\{e\}\times S^1$-action on $\C^n$ is given by
	$$t\cdot (z_1,\hdots,z_n)=(z_1,\hdots,z_{k-1},t^{m'} z_k,t^{l'}z_{k+1},z_{k+2},\hdots,z_n)$$
	for natural numbers $m'$ and $l'$ such that $ml'-m'l=\pm 1$.\\
 In particular, if we denote by $\mathcal{X}$ the fundamental vector field of that $S^1$-action, then there are rational constants $C_1$ and $C_2$ such that $\mathcal{X}=C_1\mathcal{X}_1+C_2 \mathcal{X}_2$, where $\mathcal{X}_1$ is the fundamental vector field of the $k$-th circle in $T$,
 and $\mathcal{X}_2$ is the fundamental vector field of a certain $S^1$-representation on the $k+1$-th coordinate
 of $\C^n$.
\end{lemma}
\begin{proof}
    An automorphism $\phi$ as mentioned in the statement is uniquely determined by what it does on $\{e\}\times S^1$.
    If the action of the latter on the first
    coordinate, for example, has weight $j$, then we want
    $\phi(e,t)$ to have first coordinate $t^{-j}$ for $t\in S^1$. Repeat this for all coordinates in $T$, except the $k$-th one. The resulting action is then of the desired form; the restriction of $m'$ and $l'$ come from the assumption that the $T\times S^1$-action is effective.
\end{proof}
\subsection{Torus manifolds and equivariant formality}
When we deal with a high amount of abelian symmetry, that is, with torus manifolds, there are particularly nice links between
the homology of its orbit space and their equivariant formality over $\Z$ due to Masuda and Panov, \cite{MP03}.
\begin{theorem}\label{torusManifold}
	The following are equivalent for a torus manifold $M$:
	\begin{itemize}
		\item The action is equivariantly formal.
		\item The action is locally standard and each face $F$ of $M^*$ as well as $M^*$ itself is acyclic.
	\end{itemize}
\end{theorem}
We now formulate the (specialization of a) key theorem
for this, which comes from \cite[Theorem 2.1]{FP07}. To formulate this, we recall the notion $T^p$ for the minimal subgroup in $T$ containing all
elements of order $p$, as well as the notion $M_{p,i}$ for the set of all $T^p$-orbits consisting of at most $p^i$ points.
\begin{theorem}\label{thm:ABFP}
	Let $M$ be a closed $(S^1)^n=T$-manifold such that the $T$-action is equivariantly formal over a coefficient ring $R\in \{\Z_p,\Z,\Q\}$.
	Assume that $M_{p,i}=M_i$ for $R=\Z_p$ and $M_{p,i-1}\subset M_i$ for all $p\neq -1,0,1$ for $R=\Z$. Then the Atiyah-Bredon sequence
	\[
	0 \to H_T^*(M)\to H_T^*(X_0)\to H_T^{*+1}(X_1,X_0)\to \hdots \to H_T^{*+n}(X_n,X_{n-1}) \to 0
	\]
	is exact, where the first map is the natural restriction, and the other maps are the boundary maps coming from
	the long exact cohomology sequence of the triple $(X_{i+1},X_i,X_{i-1})$.
\end{theorem}
\begin{remark}\label{rem:ABFP}
	In the setting of \cref{lem:isotropycomplone}, we conclude that the condition $M_{p,i-1}\subset M_i$ always holds for those manifolds, since
	there is at most one cyclic subgroup in $T_x$ for all $x\in M$, so that $T_x$ is contained in
	a subtorus of dimension $\dim(T_x)+1$ (see \cite[Corollary 2.2]{FP07} ff.), or connected.\\
	Likewise, the conditions $M_{p,i}=M_i$ respectively $M_{p,i-1}\subset M_i$ are clearly fulfilled whenever a $T^n$-action on a $2n$-dimensional
	manifold is locally standard.
\end{remark}
Using \cref{thm:ABFP} as well as the same ideas as in \cite{AM23}, it is straightforward to show the following:
\begin{theorem}\label{torusManifoldgen}
	Let $R$ be either $\Q$ or $\Z_p$, $p$ a prime. Then any face $F$ of $M^*$ as well as $M^*$ itself is acyclic over $R$ for an equivariantly formal (and locally standard, if $R\neq \Q$) torus manifold $M$.
\end{theorem}
\begin{proof}
	For $R=\Q$, this is really just the proof of \cite[Proposition 2.4]{AM23} carried out for the torus having
 dimension $n$ as opposed to $n-1$. For $R=\Z_p$,
    we can do the same thing (using the version of \cref{thm:ABFP} for $R=\Z_p$); we only need to check in addition that $M^H$ for a subtorus $H\subset T^n$ of dimension $1$
	(and thus any subtorus by induction) is equivariantly formal.\\
	Since we assumed that the action is locally standard in this case, to any connected component $C$ of $M^H$ there is a subgroup $G$ of order $p$ such that $C$ is a connected component of $M^G$.
	Therefore, it suffices to show that $H^{odd}(M^G;\Z_p)=0$. But it is standard that
	\[
	\dim H^{odd}(M^G;\Z_p)\leq \dim H^{odd}(M;\Z_p)=0.
	\]
\end{proof}

\newpage

\section{Orientability of GKM graphs and smooth neighborhoods of the equivariant one skeleton}\label{secorient}
Suppose an abstract GKM graph $\Gamma$ is given. If this was the GKM graph of a GKM manifold, then a small equivariant neighborhood
would have to be orientable as an open manifold. Here, we will construct to 
every GKM graph $\Gamma$ a smooth (not necessarily orientable) manifold with boundary $M_1'$ that models this small neighborhood, together
with an equivariant deformation retract $M_1'\to M_1$, where we mean by $M_1$ the actual equivariant one skeleton corresponding to $\Gamma$,
and then investigate some conditions under which we can assure that $M_1'$ becomes orientable.\\
This is partly done in \cite{GKZ22}, and also in \cite{GZ01} for an 'open thickening'.
\begin{remark}\label{X_1^*}
	We will associate a $T^k$-manifold $M_1'(\Gamma)$ (this can be seen as an 'equivariant tubular neighborhood' of the equivariant one-skeleton) with boundary $X_1(\Gamma)$ to a GKM graph $\Gamma$.
	The construction will depend on certain choices, for example the connection on $\Gamma$. We will deal with this later.\\
	For a fixed element $p$ in $V(\Gamma)$, we denote by $\C^n_p$ a representation
	of $T^k$ on $\C^n$ according to the labels of the edges emerging at $p$ and by $S(p),D(p)\subset \C^n_p$ the unit sphere respectively the unit disc
	(this corresponds to choosing signs for the labels). Let $T'$ be some maximal tree of the graph. Whenever two vertices $p_1$ and $p_2$ are connected by an edge in $T'$, then we consider
	the equivariant connected sum of $S(p_1)$ and $S(p_2)$ along their shared invariant subcircle, which means that we take out a neighborhood of this $S^1$ in both
	$S(p_1)$ and $S(p_2)$ and glue
	the spaces along the boundaries $S^1\times S^{2n-3}\subset S^{2n-1}\subset \C\times \C^{n-1}$ with a $T^k$-equivariant diffeomorphism that restricts to a linear isomorphism $h_{(p_1,p_2)}$ on $\{e\}\times S^{2n-3}$
	which sends $S^1\subset S^{2n-3}$ corresponding to an edge $e$ at $p_1$ to $S^1\subset S^{2n-3}$ corresponding to the edge $\nabla_{(p_1,p_2)} e$
	(this is well-defined due to the compatibility condition between connection and labeling of the graph). This will be the boundary of the space
	\[
	M'=((D(p_1)\setminus S^1) \amalg (D(p_2)\setminus S^1))/\sim,
	\]
	where we identify those two in an open neighborhood of the $S^1$'s we take out (this open neighborhood minus $S^1$ is equivariantly diffeomorphic to $S^1\times D^{2n-2}\times (0,1]$,
	so we can identify those neighborhoods in the same way as for the $S(p_i)$ before). Also, there is a natural map $r_1$ from $M'$ to the $T$-invariant sphere
	inside it, which is an equivariant deformation retract. This comes from the natural deformation retracts from the $D(p_i)$ to $0$. Indeed, we can deform these
	on $D(p_i)\setminus S^1$ only in a neighborhood of $S^1$ as indicated in \cref{fig:defretr} and now it is clear that this extends to the desired map $r_1$ on $M'$ canonically.
	Note further that the homotopy from the identity to $r_1$ restricted to $\partial M'$ defines a collar of $\partial M'$, and that the image of $\partial M'$ under the whole
	homotopy intersects itself only at the very beginning.
	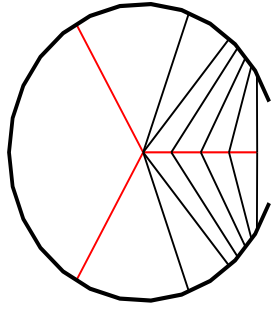
\begin{figure}[ht] 
		\centering
		\resizebox{4cm}{4cm}{
			\begin{tikzpicture}
				\draw[red] (0,0) -- (120:1cm);
				\draw[red] (0,0) -- (240:1cm);
				\draw[red] (0,0) -- (0:0.85cm);
				\draw (0,0) -- (50:1cm);
				\draw (0,0) -- (70:1cm);
				\draw (0,0) -- (310:1cm);
				\draw (0,0) -- (290:1cm);
				\draw (0.21,0) -- (45:1cm);
				\draw (0.43,0) -- (40:1cm);
				\draw (0.64,0) -- (36:1cm);
				\draw (0.85,0) -- (32:1cm);
				\draw (0.21,0) -- (-45:1cm);
				\draw (0.43,0) -- (-40:1cm);
				\draw (0.64,0) -- (-36:1cm);
				\draw (0.85,0) -- (-32:1cm);
				\draw [thick,domain=20:340] plot ({cos(\x)}, {sin(\x)});
			\end{tikzpicture}
		}
		\caption{The deformation for $n=3$. The red lines represent the one-skeleton, and the black lines the map.  \label{fig:defretr}}
	\end{figure} 
	
	Doing this for all points in $T'$ we obtain a simply-connected $T^k$-manifold $M_1'(T')$ with boundary $X_1(T')$ and the map $r_1(T')$ to its equivariant one-skeleton.
	Now we take an edge $e\in \Gamma \setminus T'$ with vertices $v_1$ and $v_2$,
	and perform the equivariant connected sum, again. Doing this for all edges in $\Gamma \setminus T$ gives us a (not necessary orientable!) $T^k$-manifold
	$M_1'$ (with the map $r_1$ to its one-skeleton, and with boundary $X_1$) whose fundamental group is isomorphic to that of $\Gamma$. We also have $H_2(X_1)\cong H_2(M_1)$ realized by $r_1$, which can be seen inductively, using
	the iterative construction of $M_1$ respectively $X_1$ and the Mayer Vietoris sequence. Indeed, when we denote by $X_1(k)$ the manifold constructed corresponding to a subtree $T_k\subset T'$ with $k$ edges
	and we assume that $H_2(X_1(k)) \overset{\cong}{\to} H_2(M_1(k))=\Z^k$ is an isomorphism, then we have the diagram
	\[
	\begin{tikzcd}
		\hdots \to H_2(X_1(k)\setminus S^1)\oplus H_2(S^{2n-1}\setminus S^1) \arrow{r} \arrow{d} &  H_2(X_1(k+1)) \arrow{r} \arrow{d} & H_1(S^1\times S^{2n-3}) \arrow{d}{\cong}\arrow{r} &\hdots \\
		\hdots \to H_2(M_1\setminus S^1)\oplus H_2(D^{2n}\setminus S^1) \arrow{r} &  H_2(M_1(k+1)) \arrow{r} & H_1(S^1)\arrow{r} & \hdots
	\end{tikzcd}
	\]
	The assertion follows because $H_2(X_1(k)\setminus S^1)=H_2(X_1(k))$, $H_2(S^{2n-1}\setminus S^1)=H_2(D^{2n}\setminus S^1)=0$ and
	\[
	H_1(S^1\times S^{2n-3})\to H_1(X_1(k)\setminus S^1)\oplus H_1(S^{2n-1}\setminus S^1)
	\]
	is the $0$-map. Similarly, we can argue for each step after $X_1(T')$ is already constructed (that is, when we start gluing $X_1(T')$ to itself).\\
	Note that the statement about the second homology groups does not depend on the choices of the $T^k$-representation on $\C^n$ made.
\end{remark}
We made some choices in the construction. In order to argue that they are not restrictive, we need an elementary lemma.
\begin{lemma}\label{lem:commuting}
	Let $A\in \text{O}(2n)$ act linearly on $\R^{2n}=(\R^2)^n=\C^n$ such that it commutes with an $S^1$-representation on $\C^n$
	that only fixes $0$. Then $A$ is contained in the standard $\U(n)\subset \SO(2n)$.
\end{lemma}
\begin{proof}
	Since $A$ and the $S^1$-representation commute, we may average a scalar product invariant under $A$ and thus get a  scalar product which is invariant under both
	$A$ and the $S^1$-representation. Thus, $\R^{2n}$ decomposes orthogonally into subrepresentations of $S^1$, indexed by their unsigned weight, none of which is $0$
	by assumption. Then $A$ leaves those subrepresentations whose weight is not $\pm 1$ invariant because it commutes with the representation,
	and thus also the orthogonal complement, the subrepresentation with weight $\pm 1$.\\
	On such a subrepresentation of complex dimension $k$, after dividing out the kernel, we simply have the diagonal $S^1$-action on $\C^k$.
	This commutes with $A$ restricted to this $\C^k$, and so $A$ acts as an element of $\U(k)$. This shows the claim.
\end{proof}
Now we can slowly go through the choices made in \cref{X_1^*}.
\begin{remark}
	In the construction, we specified how the linear isomorphism $h=h_{(p_1,p_2)}$ has to look like, but actually any choice
	of linear isomorphism $g$ with respect to which the gluing map $S^1\times S^{2n-3}\to S^1\times S^{2n-3}$ becomes $T^k$-equivariant
	will be homotopic to linear isomorphisms with the same property (which makes the resulting manifolds equivariantly diffeomorphic).
	To see this, we view $g$ as a map $\C^{n-1}\to \C^{n-1}$ and denote by $T'$ the kernel of $\alpha=\alpha((p_1,p_2))$.
	This is of the form $T\times Z$, where $T$ is a $k-1$-dimensional torus and $Z$ a cyclic subgroup (non-trivial if and only if $\alpha$
	is not primitive). Now $T'$ acts on both copies of $\C^{n-1}$, and similar to the proof of \cref{lem:commuting}, we have a scalar product
	invariant under $T'$ and $h^{-1}\circ g$ and an orthogonal decomposition
	of any such $\C^{n-1}$ into subspaces on which $T'$, after dividing out the kernel, acts as the diagonal $S^1$-action. Hence, both $g$ and $h$ restricted to
	any such $m$-dimensional subspace $S$ will be equivariantly homotopic to eachother,
	because $h^{-1}\circ g$ centralizes a subcircle $S^1\subset \U(m)$ and is thus contained in $\U(m)$ by \cref{lem:commuting}, and
	homotopic to the identity through elements in $\U(m)$ centralizing this $S^1$.
\end{remark}
\begin{definition}\label{orientdef}
	We call the GKM graph $\Gamma$ \textit{orientable} if there exists a choice of representations $\C^n_p$, $p\in V(\Gamma)$, and a connection $\nabla$ on $\Gamma$ such that $X_1(\Gamma)$ as in \cref{X_1^*}
	is orientable.
\end{definition}
\begin{remark}
	This does not actually depend on the choice of connection. Indeed, a different choice of connection on an edge $(p_1,p_2)$, while fixing all other choices, only gives rise to a different map
	$h'_{(p_1,p_2)}$, which, by the last remark, is equivariantly homotopic to $h_{(p_1,p_2)}$ (the map from the 'old' connection) through linear isomorphisms.
\end{remark}
\begin{remark}\label{equneigh}
	Note that this was a construction based on an abstract GKM graph. However, given a GKM manifold $M$ whose graph $\Gamma$ is orientable,
	we will show that, with the choices of signs coming from the orientability, $M_1'(\Gamma)$ can be embdedded equivariantly (non-smoothly, up to now)
	into $M$. Indeed, $M_0'$ can clearly be embedded equivariantly (with those choices of signs), and we may modify this embedding in such a way that, whenever $(p_1,p_2)$ is an edge,
	$D(p_1)$ and $D(p_2)$ touch precisely in neighborhoods $\mathcal{N}_1(S^1)\subset D(p_1)$ respectively $\mathcal{N}_2(S^1)\subset D(p_2)$ of the shared subcircle
	in their boundaries and the induced map
	$$S^1\times D^{2n-2}=\mathcal{N}_1(S^1)\to \mathcal{N}_2(S^1)=S^1\times D^{2n-2}$$
	is an equivariant fiberwise linear isomorphism. The image of these $D(p_i)$ under this modification is not necessarily a smooth manifold, but a topological manifold
	with boundary equivariantly homeomorphic to $M_1'(\Gamma)$ (the smoothness will be done later in \cref{equneigh}).
\end{remark}
\begin{remark}\label{singfol}
	There is a natural singular foliation $\mathcal{F}_1$, whose leaves are tori of different dimension, on $M_1'$. This is given on $M_0'$ by the orbits
	of the natural $T^n$-action on $D^{2n}$, and this is preserved under the gluing maps used in \cref{X_1^*} (although the $T^n$-action might not be).
	We will denote by $Y_1 \subset X_1$ the leaves of maximal dimension $n$. This has the natural structure of a (not necessarily principal) $T^n$-bundle over a space $B_1$ which is homotopy
	equivalent to $\Gamma$, because, via the construction in \cref{X_1^*}, $B_1$ is obtained by gluing disks (namely the orbit spaces of the free $T^n$-orbits of the natural action on $S^{2n-1}$)
	onto each other along smaller disks. 
\end{remark}
\begin{remark}
	There is also the definition of an oriented graph as in \cite{BP15}[Definition 7.9.16]. It is unknown to the author how these definitions are connected.
\end{remark}
\begin{definition}
	The GKM graph $\Gamma$ is called $j$-independent, $j\geq 2$, if for any vertex any $j$ labels are linearly independent over $\Q$.\\
	If $j=k=n-1$, then we say that the graph is in \textit{general position}.\\
	If $j=k=n$, then we speak of a \textit{torus graph}.
\end{definition}
Note that the following holds for any $j$-independent GKM graph $\Gamma$: at any vertex, any $k<j$
edges determine a unique ($k$-independent) GKM subgraph (see \cite[Proposition 5.7]{AC22},
for example).
\begin{definition}\label{def:connectionpath}
	A closed simple (that is, edges occur at most once) edge path
	$$e_1=(p_1,p_2),e_2=(p_2,p_3),\hdots,e_{n-1}=(p_{n-1},p_{n}), e_n=(p_n,p_1), e_{n+1}=e_1,\hdots$$
	is called a \textit{connection path} of the GKM graph $(\Gamma,\nabla)$ if $\nabla_{e_{i-1}} e_i=e_{i+1}$ for $i\in \{1,\hdots,n\}$.
\end{definition}
\begin{remark}\label{orientable}
	There are two basic properties of orientability of graphs that we want to mention now.
	\begin{enumerate}
		\item A subgraph $\Gamma'$ of an orientable torus graph $\Gamma$ is orientable. Indeed, the manifold $X_1(\Gamma')$ is a connected component of
		$X_1(\Gamma)^H$ for a closed subgroup $H\subset T^n$, and the normal bundle of $X_1(\Gamma)^H$ is a sum of line bundles, hence orientable.
		\item If the fundamental group of a 3-independent graph is generated by connection paths, then it is orientable. The reason is that the equivariant two-skeleton $(X_1)_2$ of $X_1$
		generates $\pi_1(X_1)$, that every connected component of $(X_1)_2$ is equivariantly diffeomorphic to $S^1\times T^2$ with an action of $T^k$ on $T^2$ induced
		by an epimorphism and that the normal bundle of such an $S^1\times T^2$ is orientable. The last statement is true for $4$-independent graphs, because
		then the normal bundle splits into line bundles. But this is not necessarily true anymore if the graph is 3-independent. In this case, however,
		there is a linear $T^{k-2}$-action on the fiber $\R^{2n-4}$ of the normal bundle that only fixes $0$ and commutes with the element
		$A\in \text{O}(2n-4)$ that defines the normal bundle of $S^1=(S^1\times T^2)/T^2$ in $X_1/T^2$. This implies that $A$ is in $\U(n-2)$ by \cref{lem:commuting}.
	\end{enumerate}
\end{remark}

\newpage

\section{Some statements about graphs}\label{graph}
The main result of this small section is \cref{ext}, which states that certain 4-independent GKM graphs are actually restrictions
of a torus graph. The main assumption is that the fundamental group of $\Gamma$ is generated by connection paths (see \cref{def:connectionpath}).
This assumption
is quite natural in view of \cref{orbitSpace}: if $\Gamma$ was realizable as an equivariantly formal (even only over $\Q$) GKM manifold $M$, then
it would follow that $b_1(M_2^*)=0$, which is equivalent to the statement that $H_1(\Gamma;\Q)$ is generated by connection paths,
because $M_2^*$ is obtained from $M_1^*\cong \Gamma$ by attaching disks only along connection paths. We will see in \cref{rem:graphbetti}
that this, in turn, is equivalent to the same statement on fundamental groups.\\
Let us begin with a more fundamental concept, for which we do not need the 4-independency yet.
\begin{lemma}\label{tree}
	Let $\Gamma$ be an $n$-valent GKM graph with the following properties:
	\begin{enumerate}
		\item Any connection path is a two-valent GKM subgraph.
		\item The group $\pi_1(\Gamma)$ is generated by connection paths.
	\end{enumerate}
	Then there is a maximal contractible tree $T\subset \Gamma$ together with an ordered tuple of edges $e_1,\hdots,e_k=E(\Gamma)\setminus E(T)$
	such that attaching $e_1$ to $T$ closes a connection path, attaching $e_2$ on $T\cup e_1$ closes a connection path, $\hdots$, attaching $e_k$
	on $\Gamma\setminus e_k$ closes a connection path.
\end{lemma}
\begin{proof}
	We denote by $\mathcal{G}_1$ the set of all connection paths of $\Gamma$, ordered increasingly by the number of their edges.
	Take any of the first connection paths $\gamma_1$ of $\mathcal{G}_1$ and remove an edge. This will be $e_1$. Now set $\Gamma_2:=\Gamma/\gamma_1$,
	and let $\mathcal{G}_2$ be the set of all non-trivial (that is, non-constant paths) descensions of elements of $\mathcal{G}_1$, ordered by number of edges.
	$\Gamma_2$ is homotopy equivalent to $\Gamma\setminus e_1$, and we claim that the elements of $\mathcal{G}_2$ generate the fundamental group of $\Gamma_2$.
	To see this, we need to check that no connection path $\gamma$ in $\Gamma_1$ except $\gamma_1$ descends to a point in $\Gamma_2$. This is true as long
	as not all edges of $\gamma$ collapse. If $\gamma$ has more edges than $\gamma_1$, this is clear. If $\gamma$ has the same amount of edges,
	then these would have been precisely those of $\gamma_1$, so $\gamma=\gamma_1$.\\	
	Now take the first element $\gamma_2'\in \mathcal{G}_2$ and remove an edge $e_2'$ of $\gamma_2'$. There is a corresponding $\gamma_2\subset \Gamma$
	(which might possibly intersect $e_1$) and a corresponding edge $e_2$ in $\gamma_2$,
	which is not contained in $\gamma_1$. Therefore, putting $e_1$ inside $\Gamma\setminus (e_1\cup e_2)$ still closes a connection path, and so does
	putting $e_2$ inside $\Gamma\setminus e_2$. Now set $\Gamma_3:=\Gamma_2/\gamma_2'$, which is homotopy equivalent to $\Gamma\setminus (e_1\cup e_2)$,
	define $\mathcal{G}_3$ as usual, and so on. Once again, $\mathcal{G}_3$ generates the fundamental group of $\Gamma_3$ by the same argument as before.\\
	We may repeat these arguments until some $\Gamma_{k+1}$ is contractible (which eventually has to happen before $\mathcal{G}_k$ becomes empty,
	because $\pi_1(\Gamma)$ is generated by connection paths). The ordered tuple $(e_1,\hdots, e_k)$ then has the desired properties.
\end{proof}
\begin{remark}\label{rem:graphbetti}
	Under the assumption that each connection path is a two-valent GKM subgraph, the assumption that $\pi_1(\Gamma)$ is generated by connection paths
	is implied by (and thus equivalent to) the assumption that $H_1(\Gamma;\Q)$ is generated by connection paths. Indeed, by the same process as in the proof of \cref{tree},
	$b_1(\Gamma_{k+1})$ eventually has to become $0$, which, since $\Gamma_{k+1}$ is a graph, implies that $\Gamma_{k+1}$ is contractible.\\
 So, going backwards from $\Gamma_{k+1}$, we see that $\gamma_1$ generates the fundamental group of $\Gamma_{k+1}\cup e_1$,
 $\gamma_1$ and $\gamma_2$ generate the fundamental group of $\Gamma_{k+1}\cup e_1\cup e_2$, and so on, since
 $\gamma_2$ does never cross $e_1$ by choice of $\gamma_2$.
\end{remark}
Consider a two-valent GKM subgraph $\gamma$ of $\Gamma$ and fix a vertex $v\in V(\gamma)$. Denote by $F(v)$ the set $E(\Gamma)_v\setminus E(\gamma)_v$.
When taking an edge in $F(v)$ and pushing this one time around $\gamma$ with the connection, we get a bijection $\mu_{\gamma}\colon F(v)\to F(v)$,
the monodromy map of $\gamma$. Note that this is trivial for all $\gamma$ if the graph is at least $4$-independent, and that in this case,
any two edges at one vertex determine a unique two-valent GKM subgraph. We want to prove the following theorem.
\begin{theorem}\label{ext}
Let $\Gamma$ be an $n$-valent $3$-independent GKM graph with the property that any three edges belong to a unique 3-valent GKM-subgraph
and that
the group $\pi_1(\Gamma)$ is generated by two-valent GKM subgraphs. Then the $\Z^k$-labeling (respectively $\Z^k/\pm 1$)
extends to an effective $\Z^n$-labeling (respectively $\Z^n/\pm 1$).
\end{theorem}
\begin{proof}
	We do this in the signed case, the unsigned case is analogous. We begin by noting three things.
	\begin{itemize}
		\item For every maximal contractible tree $T\subset \Gamma$, there is an extension by extending at
		one vertex and then pushing this over whole $T$ via the connection.
		\item By the same reasoning, whenever there is an extension, it is uniquely determined by the extension at one vertex.
		\item At any vertex, any three edges $e_1$, $e_2$ and $e_3$ define a unique 3-valent, 3-independent subgraph $\Gamma'$ by assumption.
		We can extend the labeling $\alpha(e_i)$ on this vertex to a $\Z^n$-labeling $(\alpha(e_i),0)\in \Z^k\times \Z^{n-k}$ and this clearly gives a well-defined
		extension on $\Gamma'$. It follows that any three labels $\alpha'(e_i)\in \Z^n$ give a well-defined $\Z^n$-labeling on $\Gamma'$,
		since the GKM condition is linear and there is an isomorphism $\Q^n\to \Q^n$ sending $(\alpha(e_i),0)$ to $\alpha'(e_i)$.\\
		In particular, for any connection path $\gamma$, there is an extension of the labeling on all edges meeting
		$\gamma$. Indeed, we may extend the $\Z^k$-labeling at some vertex to an effective $\Z^n$-labeling. Now for some vertex $v$ of $\gamma$, we take some edge $e$ at $v$ transverse
  to $\gamma$, and this gives an extension on all edges that meet $\gamma$ \textbf{and} are contained
  in the unique 3-valent GKM subgraph $\Gamma'$ spanned by $\gamma$ and $e$. Now take some other
  transverse edge $e'\neq e$ at $v$ and repeat the argument, then take a transverse edge $e''\neq e,e'$ at $v$ and repeat the argument, and so forth. This will give a well-defined labeling near
  $\gamma$, since the intersection of two of these 3-valent subgraphs near $\gamma$ is precisely $\gamma$.
	\end{itemize}
Now choose a maximal contractible tree $T$ as in \cref{tree} (which is possible, because any two adjacent edges belong to a unique
two-valent GKM subgraph by our assumption), choose an extension at some vertex of $T$ and consider
the induced extension on whole $T$. Normally, when attaching $e_1$, this extension of $T$ might not be compatible.
But here, we close a connection path $\gamma_1$ when attaching $e_1$, so by the third bullet above, the GKM-condition
at each edge of $\gamma_1$ is indeed fulfilled, so in particular at $e_1$. When attaching $e_2$, we close a connection path, again, so the GKM condition is fulfilled,
and so on. This shows that the GKM condition holds everywhere on $\Gamma$, which we wanted to prove.
\end{proof}
In particular, every realizable GKM graph which is at least $4$-independent comes from an $n$-independent GKM graph, since for
GKM$_4$-manifolds $M$ we have $b_1(M_2^*)=\pi_1(M_2^*)=0$, so $M_1^*=\Gamma$ has its fundamental group generated by two-valent subgraphs.

\newpage

\section{Technical properties of manifolds in general position}\label{local}
In this section, we establish some results regarding the homological behaviour of a (possibly non-compact) $T$-manifold $M$ near the equivariant $k$-skeleton $M_k$, for $k\leq n-2$.
We assume that the action of $T$ on $M$ is in general position, that any $j\leq n-2$ edges emerging from a vertex belong to a $T$-invariant, equivariantly formal
(coefficients $R$ are taken to be $\Z$ or $\Q$) face submanifold of dimension $2j$, and that $M_k^*$ is $k-1$-acyclic, $k\leq n-2$. In particular
$b_1(M_2^*)=0$ and thus $\pi_1(\Gamma)$ is generated by connection paths by \cref{rem:graphbetti}.\\
While $M$ is automatically orientable then for $R=\Z$ (by the universal coefficient theorem we have $H^1(M;\Z_2)=\Hom_{\Z}(H_1(M;\Z),\Z_2)=0$),
we need to additionally assume $M$ to be orientable for $R=\Q$.\\

We set $F_{k+1}$ to be $M_{k+1}$ minus an open equivariant neighborhood of $M_k$ in $M_{k+1}$. Every connected component of this belongs
to a face submanifold $Z$ of dimension $2k+2$, hence to a unique GKM subgraph of valence $k+1$, and can be considered to be the complement of a small 
equivariant neighborhood of $Z_k$ in $Z$.\\
There are two settings now: $R=\Z$ and $R=\Q$. The reason why these are different is that for $R=\Z$,
the action of $T$ on every such $Z\setminus Z_k$ is free after dividing out the kernel (by \cref{torusManifold}).
This also implies that $F_{k+1}$ is equivariantly diffeomorphic to the product of a homology disk and $T^{k+1}$, where $T$ acts on $T^{k+1}$ via some surjectice homomorphism.
Likewise, any connected component of the boundary of $F_{k+1}$ is the product of a homology sphere and $T^{k+1}$.\\
For $R=\Q$, the equivariant formality does not imply that the action on $Z$ is locally standard, of course. There might be many disconnected stabilizers occuring,
which needs to be respected in the local description of $M$ around its equivariant $k$-skeleton, $k\leq n-2$. If, in the following, a statement is written without
specifying the coefficients, then it holds for both $\Q$ and $\Z$.
\begin{remark}\label{rem:boundaryidentification}
	In both cases, we may consider the boundary of $F_{k+1}$ both as a subset of $F_{k+1}$ and of the boundary of a small $T$-invariant neighborhood $M_k'$ of $M_k$ in $M$
	intersected with $M_{k+1}$. We write $\partial F_{k+1}$ for the boundary considered in $F_{k+1}$ and $\mathcal{M}_k$ for the boundary
	considered in $M_k'$. Though they are named differently, these are the same sets, and there is the 'natural identification' $\partial F_{k+1}\cong \mathcal{M}_k$ given by $p\mapsto p$.\\
	Although it seems to be redundant, it will be convenient to use these two different notions for the same set. We will come back to this later.
\end{remark}
\begin{lemma}\label{equlinebundles1}
	Let $n\geq 4$. For $k\leq n-3$, the normal bundle of any $2k$-dimensional face submanifold $Z$ splits equivariantly into line bundles (in particular, $Z$ is orientable also for $R=\Q$).\\
	Specifically for $R=\Z$, the normal bundle restricted to $Z\setminus Z_{k-1}$ is equivariantly trivial (that is, a product space endoweed with a product action),
	and for $Z$ of dimension $2n-4$, the normal bundle of $Z$ restricted to the boundary of an equivariant neighborhood of $Z_{n-3}\subset Z$ is equivariantly trivial.
\end{lemma}
\begin{proof}
	$Z$ is contained in $n-k$ face submanifolds of dimension $2(k+1)$, so the normal bundle of $Z$ in all of these is an equivariant line bundle.
	This implies the first claim.\\
	For $R=\Z$, the bundle (which is a sum of equivariant line bundles) is equivariantly trivial over $Z\setminus Z_{k-1}$ if
	every line bundle is equivariantly trivial over $Z\setminus Z_{k-1}$, and the latter is true if and only if the corresponding $S^1$-bundle
	$S^1\to E\to Z\setminus Z_{k-1}$ is equivariantly trivial over $Z\setminus Z_{k-1}$. By \cite[Theorem A]{LMS83}, this is true if and only if the equivariant first Chern class $c_1$
	of this bundle vanishes. Since the $T$-action on $Z\setminus Z_{k-1}$ is free after dividing out the kernel and the orbit space is a homology disk,
	$c_1$ is determined by its restriction to $T^k\times_T BT$. But the $S^1$-bundle restricted to a $T$-orbit in $Z\setminus Z_{k-1}$
	is equivariantly trivial, because the total space itself is a $T$-orbit, so we are done.\\
	The third claim will be treated soon in \cref{equlinebundles2}.
\end{proof}
In the next lemma, we establish the existence of something like an 'equivariant tubular neighborhood' of the equivariant $k$-skeleton in $M$.
\begin{lemma}\label{defretr}
	Let $n\geq 3$. For all $k\leq n-2$, there is a closed equivariant deformation retract $r_k\colon M_k'\to M_k$ in $M$ whose boundary $X_k$ is a smooth manifold.
	Moreover, $r_{k+1}$ restricted to $M_k'$ minus a neighborhood of $\mathcal{M}_k=(M_{k+1}\cap X_k)$ may be assumed to be equivariantly homotopic
	to $r_k$ and, when restricted to a suitable embedding of
	the normal bundle of $F_{k+1}$ in $M_{k+1}'\setminus (M_k'\setminus X_k)$, it may be assumed to collaps the fibers.\\
	Further, the flow $F_{[0,1]}(X_k)$ of $X_k$ under the whole homotopy intersects $X_k$ only at $t=0$,
	and $F$ restricted to $[0,\eps]\times X_k$ may be assumed to be smooth with $\partial_{t=0} F_t$ being a vector field transverse to $X_k$.\\
	If the $T$-action comes from a $T^n$-action, then we may take everything equivariant with respect to this $T^n$-action.
\end{lemma}
\begin{proof}
	This exists for $k=0$ as already mentioned in \cref{X_1^*}.\\
	Assume $r_k$ and an equivariant homotopy $F_t$ from the identity to $r_{k}$ exists for some $k\leq n-3$. In order to construct this for $k+1$, we need to
	construct $r_{k+1}$ on $M_{k}'$, first.\\
	By \cref{equlinebundles1}, a small, closed equivariant neighborhood of $\mathcal{M}_k$ in $X_k$ is equivariantly diffeomorphic to $\mathcal{M}_k\tilde{\times} D^{2(n-k-1)}$
	(and we fix this identification together with a fiber metric $g$ from now on), where we just mean that this is an equivariant bundle over $\mathcal{M}_k$ (which splits equivariantly into line bundles for $k\leq n-4$).
	The complement $N_k$ of the interior of $M_k'$ in $M$ is a manifold with boundary $X_k$, and its equivariant $k+1$-skeleton is
	precisely $F_{k+1}$. Again, a neighborhood of $F_{k+1}$ in $N_k$ is equivariantly diffeomorphic to $F_{k+1}\tilde{\times} D^{2(n-k-1)}$. 
 By possibly shrinking the neighborhood, we may assume that under the identification in \cref{rem:boundaryidentification}
	$\partial F_{k+1}\tilde{\times} D^{2(n-k-1)}$ is even mapped into $\mathcal{M}_k\tilde{\times} D^{2(n-k-1)}$. It is now standard differential topology
	that this emebdding can be isotoped rel $\partial F_{k+1}\tilde{\times} \{0\}$  equivariantly into a fiberwise linear embedding and maps $\partial F_{k+1}\tilde{\times} D^{2(n-k-1)}$
	into the set $g(v,v)\leq \beta$ for some $\beta>0$. Then we can rescale
	those fibers to make this embedding an isomorphism
    $$f_k\colon \partial F_{k+1}\tilde{\times} D^{2(n-k-1)} \to \mathcal{M}_k\tilde{\times} D^{2(n-k-1)},$$
    which, in a local section, is of the form
	\[
	f_k(p,v):=(p,h_k(p)(v)),
	\]
	where $h_k(p)\in \SO(2(n-k-1))$ is an invariant map with the additional property that $h_k(p)$ has to commute with the action of $T_p$, the stabilizer
	of $p$, on the $D^{2(n-k-1)}$-fiber over $p$.\\	
	
	By the inductivity assumption, a neighborhood $V$ of $\mathcal{M}_k$ in $M_k'$ is equivariantly diffeomorphic to $\mathcal{M}_k\tilde{\times} D^{2(n-k-1)}\times [0,\eps]$ via
	the flow of $\mathcal{M}_k\tilde{\times} D^{2(n-k-1)}$ under $F_t$. We may assume, by choosing all neighborhoods small enough,
	that $F_t(X_k)$ hits $V$ only for $p\in \mathcal{M}_k\tilde{\times} D^{2(n-k-1)}\subset X_k$ and $t\leq \eps$.\\
	Assuming that the disk has radius $1$,
	we consider the subset $U$ in there consisting of all points $(p,v,t)$ that satisfy $1/2\leq ||v||+\eps^{-1} t\leq 1$ (the red lines in \cref{fig:defretr1}).	
		\begin{figure}[ht] 
		\centering
			\begin{tikzpicture}
			\path[draw] (0,0)--(3,0)--(3,3)--(0,3)--cycle;
			\foreach \x in {0,...,20}{
				\draw[red] (1.5*\x*0.05,0)--(3,3-1.5*\x*0.05);}
			\foreach \x in {0,...,20}{
			\draw (1.5*\x*0.05,0)--(1.5*\x*0.05,3);}
			\foreach \x in {0,...,20}{
			\draw[green] (1.5+1.5*\x*0.05,0)--(3,1.5-1.5*\x*0.05);}
			\end{tikzpicture}
		\caption{\label{fig:defretr1}}
	\end{figure} 
	This is equivariantly diffeomorphic to $\mathcal{M}_k\tilde{\times} D^{2(n-k-1)}\times [1/2,1]$ (the black lines in \cref{fig:defretr1}).
	In fact, there is a straighforward isotopy $j_t$ rel $\mathcal{M}_k\times \{0\}\times [1/2,1]$ from the embedding
	$\mathcal{M}_k\tilde{\times} D^{2(n-k-1)}\times [1/2,1]\hookrightarrow \mathcal{M}_k\tilde{\times} D^{2(n-k-1)}\times [0,1]$ into $U$.\\
	Likewise, $U'$ consisting of all points with $0<||v||+\eps^{-1} t\leq 1$ (the union of the red lines and the green lines) is equivariantly diffeomorphic to $\mathcal{M}_k\tilde{\times} D^{2(n-k-1)}\times (0,1]$,
	and we fix this identification from now on.\\
	Choose a bump function $\rho^1\colon [0,1]\to [0,1]$ identically $0$ on $[0,1/2+\delta]$ and identically $1$ on $[1/2+ 2\delta,1]$ for some small $\delta>0$.
	Moreover, choose a bump function $\rho^2 \colon [0,1]\to [0,1]$ identically $1$ near $1$ and identically $0$ on $[0,3/4]$.
	Define
	\[
	\rho_1\colon (0,1]\to [0,1], \; \rho_1(x):=\rho^1(x)/x, \quad \rho_1\colon (0,1]\to [0,1], \; \rho_2(x):=\rho^2(x)/x
	\]
	For $x\in \mathcal{M}_k$ we define $r_{k+1}(x):=x$, for $x=(p,v,t)\in U'$ we define
	\[
	r_{k+1}(x):=F_{\rho_1(||v||+\eps^{-1}t)}(p,\rho_2(||v||+\eps^{-1}t) v, t+\eps(1-\rho_2(||v||+\eps^{-1}t))||v||),
	\]
	and for $x$ not in $U'$ we set $r_{k+1}(x):=r_k(x)$ (that is, we interpolate between $r_k$ way outside $\mathcal{M}_k$ and the collapse of the $D^{2(n-k-1)}$-fibers
	corresponding to the green lines in \cref{fig:defretr1}).\\
	The image of $r_{k+1}$ is in $M_{k+1}$. There is an equivariant homotopy $G_s$ from the identity to $r_{k+1}$ which is given by
	$F_s$ outside $U'$ and inside $U'$ by
	\[
	G_s(y)=F_{s\rho_1(||v||+\eps^{-1}t)}(p,(1-s(1-\rho_2(||v||+\eps^{-1}t))) v, t+\eps s(1-\rho_2(||v||+\eps^{-1}t))||v||)
	\]
 (that is, we interpolate between the homotopy $F_s$ way outside $\mathcal{M}_k$ and the continouus shrinking of the $D^{2(n-k-1)}$-fibers	corresponding to the green lines in \cref{fig:defretr1}).
	Moreover, by the choice of $V$ above, this homotopy can be restricted to
	$M^k$ (by which we mean $M_{k}'$ minus the green region in \cref{fig:defretr1}), and so $r_{k+1}$ may be assumed to be equivariantly homotopic to the
	identity on $M^k$ as well. It is clear that $\partial_{s=0} G_s$ is transverse to $X_k\cap M^k$.\\
	
	We now define the space $M_{k+1}'$ to be $F_{k+1}\tilde{\times} D^{2(n-k-1)}$ attached to $M^k$
    along $U\cong \mathcal{M}_k\times D^{2(n-k-1)}\times [1/2,1]$ using $f_{k}$. It is straightforward to see that this embeds into $M$ in the desired way.
	The map $r_{k+1}$ and the homotopy extend to $F_{k+1}\tilde{\times} D^{2(n-k-1)}$ and thus $M_{k+1}'$ naturally, and the homotopy defines a collar of $X_{k+1}$ in $M_{k+1}'$.\\

	Now, $r_{k+1}$ restricted to $M_{k+1}$ is not the identity, but it is the identity on $M_k$ and equivariantly homotopic to the identity
	on $M_{k+1}$ with support in a neighborhood of $M_k$. Since $M_{k+1}$ is an equivariant subcomplex of the equivariant CW-complex $M_{k+1}'$,
	this homotopy lifts and we may take this new map as our $r_{k+1}$, which then fulfills all properties we wanted it to have.
\end{proof}
\begin{corollary}\label{cor:defretrnaturality}
	The deformation retract $r_k\colon M_k'\to M_k$ is natural with respect to subspaces in the sense that for a face submanifold $Z\subset M$,
	we have that the restriction $r_k\colon M_k'\cap Z\to M_k\cap Z$ is well-defined and has all the properties that the map $r_k$ has.
\end{corollary}
\begin{proof}
	This is an immediate consequence of the equivariance of the map $r_k\colon M_k'\to M_k$ as well as the equivariance of its homotopy to the identity.
\end{proof}
\begin{remark}\label{rem:identifications}
	We have fixed some identification of a small neighborhood of $\mathcal{M}_k$ in $X_k$ with $\mathcal{M}_k \tilde{\times} D^{2(n-k-1)}$,
	and relative to this we have the isomorphism $f_k$ as in the last proof. Whenever we have a $T^n$-action on $M_k'$,
	we choose this identification $T^n$-equivariant, though.\\
	We will show soon that the $T$-action on $M_k$ always comes from a $T^n$-action for $k\leq n-3$, and then we
	fix this identification to be $T^n$-equivariant (and refer to $f_k$ as the isomorphism relative to this fixed identification).
\end{remark}
\begin{remark}\label{X_1*}
	In our special case of the graph being in general position and the action being locally standard (\cref{lem:approp}),
	the orbit space $X_1^*$ is homeomorphic to $\#_k S^1\times S^{n-1}$ (where $k$ is the rank of $\pi_1(\Gamma)$) and we can choose $k$
	embedded circles in the free stratum as a basis of $\pi_1(X_1^*)$. Indeed, the orbit space of $X_1(T')$ (see \cref{X_1^*} for the notation), is homeomorphic to $S^n$,
	and for every edge $e\in \Gamma \setminus T'$, we may choose a path connecting the two $T$-fixed points in $X_1(T')$ corresponding to $e$
	whose interior is in the free stratum. The connected sum of $X_1(T')$ along these points now gives $X_1$ together with $k$ subcircles
	in the orbit space of the free stratum forming a basis of $\pi_1(X_1)=\pi_1(X_1^*)$. That the $k$-fold connected sum of $S^n$ with itself
	gives $\#_k S^1\times S^{n-1}$ is standard, because a connected sum of any manifold $M=M^n$ with itself is the connected sum of
	$M$ with $S^1\times S^{n-1}$ (write $M$ as $M \# S^n$ and perform the sum on two points of $S^n$).
\end{remark}
\begin{remark}\label{rem:orbitspacemanifold}
	For $R=\Z$, it is clear from the construction in \cref{defretr} that if the $T^{n-1}$-action on $M_k'$ is locally standard (which is true for $k=0$), then the $T^{n-1}$-action
	on $M_{k+1}'$ is locally standard, so that all orbit spaces $X_k^*$ are topological manifolds for $R=\Z$ by \cref{lem:approp}.\\
	For $R=\Q$, these are orbifolds. Indeed, $X_1^*$ is a manifold, and for a connected component $C$ of $F_{k+1}$ and its principal isotropy subgroup $T'$,
	$C\tilde{\times} D^{2(n-k-1)}/T'=C\tilde{\times} (D^{2(n-k-1)}/T')=C\tilde{\times} D^{n-k+1}$ is a manifold, and this makes
	$C\tilde{\times} D^{2(n-k-1)}/T=(C\tilde{\times} D^{n-k+1})/(T/T')$ into an orbifold (the action of $T/T'$ is almost free on the latter space). Since $X_{k+1}^*$ is obtained from $(X_k\setminus \mathcal{M}_k)^*$
	by smoothly (in the 'total space') attaching $F_{k+1}\tilde{\times} D^{2(n-k-1)}$, $X_{k+1}^*$ is an orbifold if $X_k^*$ is.
\end{remark}
\begin{lemma}\label{fol}
	Let $n\geq 5$. For $k\leq n-3$, the $T=T^{n-1}$-action on $M_k'$ extends to an effective $T^n$-action such that
	every $T^n$-orbit of dimension at most $n-2$ is a $T$-orbit.\\
    If every subgraph of covalence one is ineffective, the statement also holds for $k=n-2$.\\
    If $n=4$ and any three edges emerging at a vertex belong to an ineffective subgraph, then the statement holds as well.
\end{lemma}
\begin{proof}
	We may extend the labeling of $\Gamma$ to a $\Z^n$-labeling by \cref{ext}, and so we have the extension of the $T^{n-1}$-action on $M_0'$.\\
	So assume we have shown the statement for some $k\leq n-4$. Consider a generic stabilizer $T'$ of the $T$-action on a connected component of $F_{k+1}$.
	By \cref{rem:isotropygenpos}, the action of its connected component on the fiber $D^{2(n-k-1)}$ of the normal bundle of $F_{k+1}$ in $M_{k+1}'$
	is in general position, in particular GKM (we have $\dim(T')\geq 2$ due to $k\leq n-4$). Therefore, the centralizer of $T'$ in $\SO(2(n-k-1))$ is in the maximal torus.
	This is also true for $k=n-3$ if every subgraph of covalence one is ineffective, since then the $S^1$-representation on $D^4$
	splits canonically in $\C_{\alpha_1}\oplus \C_{\alpha_2}$, where $\alpha_i\neq 0,\pm 1$.\\
	Since for any point $p$ in the latter connected component of $F_{k+1}$ the stabilizer $T_p$ of $p$ contains that $T'$ and acts on $D^{2(n-k-1)}$,
	in particular commuting with the $T'$-representation, $T_p$ is contained in the maximal torus of $\SO(2(n-k-1))$. This shows that
	any equivariant bundle automorphism in a local section is of the form $(p,v)\mapsto h_k(p)(v)$, where $h_k(p)$ is in the maximal torus
	of $\SO(2(n-k-1))$, and in particular, that the $T$-action
	can be extended (though not necessarily effectively) on $F_{k+1}\tilde{\times} D^{2(n-k-1)}$ by letting an additional $S^1$-factor act on the fiber, arbitrarily.\\
	
	Conversely, we look at the way the $T^n$-action on $\mathcal{M}_k\tilde{\times} D^{2(n-k-1)}$ in $X_k$ is obtained from the $T$-action there.
	Around every point $p$ with generic stabilizer, a $T$-neighborhood is isomorphic to $(0,1)^k\times T\times_{T_p} \C^{n-k-1}$,
	and neither the subgroup $T_p\subset T$ nor its representation on $\C^{n-k-1}$ depend on the generic $p$ chosen.
 Moreover, the transition functions between these trivializations preserve any torus action, which implies that
	\cref{lem:extcomplone} can be applied to this setting;
	so consider the vector field $\mathcal{X}:=C_1 \mathcal{X}_1+C_2\mathcal{X}_2$ as in \cref{lem:extcomplone} on whole $F_{k+1}\tilde{\times} D^{2(n-k-1)}$, which agrees with the fundamental vector field of the $\{e\}\times S^1$-action
    on generic points of $\mathcal{M}_k\tilde{\times} D^{2(n-k-1)}$, hence everywhere. This shows that we can extend this
    fundamental vector field to $F_{k+1}\tilde{\times} D^{2(n-k-1)}$. Again, on generic points of the latter, this extension
    integrates to a periodic flow, hence it does so everywhere.\\
    
	We have thus found an extension of the $T$-action as desired.
\end{proof}
\begin{remark}\label{rem:orbitspacemanifold2}
	In the situation of \cref{fol}, denote by $Y_k$ all points through which the $T^n$-orbit has dimension $n$. By the same reasoning as in
	\cref{rem:orbitspacemanifold}, $(X_k\setminus Y_k)^*$ (orbit space taken with respect to the $T$-action, not the $T^n$-action)
	is a manifold for $R=\Z$ and an orbifold for $R=\Q$, relying on the fact that this statement is true for $k=0$.
\end{remark}
\begin{remark}
	We assumed in the proofs of \cref{defretr} and \cref{fol} that $M$ is smooth. However, this assumption is not needed necessarily. It is only needed that
	a neighborhood of each $\mathcal{M}_k$ in $X_k$ is equivariantly homeomorphic to a sum of equivariant line bundles over
	$\mathcal{M}_k$ and that the attaching maps $f_k$ are of the above mentioned form. This might be important
	when one tries to construct these spaces $M_k'$ from an abstract graph, without a manifold $M$ present (although this won't happen, here).
\end{remark}
\begin{remark}\label{equlinebundles2}
	Since we know now that the $T$-action on $M_{n-3}'$ extends to a $T^n$-action, the normal bundle of a $2n-4$-dimensional face submaniold $Z$ restricted to $Z_{n-3}$
	splits equivariantly into line bundles regardless of the assumption on subgraphs of covalence one as for example in \cref{fol}. In particular, the normal bundle of
	$\mathcal{M}_{n-3}$ splits equivariantly into line bundles. If $R=\Z$, then this bundle is equivariantly trivial, since the $T$-equivariant line bundle
	extends to a $T$-equivariant $\C^2$-bundle
	(this is indeed a $\C^2$-bundle and not only an $\R^4$-bundle, because the structure group centralizes a subtorus of the maximal torus in $\SO(4)$)
	over $Z\setminus Z_{n-3}$, and hence its first Chern class vanishes. That it is equivariantly trivial follows as in \cref{equlinebundles1}.
\end{remark}
We assume now that we are in the situation of \cref{fol}. We denote by $\mathcal{F}^j_k$ the set of all points in $T^n$-orbits (see \cref{fol})
in $X_k$ whose dimension is $j$, by $X_k^j$ the union of all $\mathcal{F}_k^j$, where $\geq j$, (of course, it is understood that $j\geq k+1$, because
the dimension of an orbit in $X_k$ is at least $k+1$) and set
$Y_k:=\mathcal{F}^n_k$. From now on, for a $T$-invariant set $O\subset M_k'$ we still mean by $O^*$ the orbit space with respect to the $T$-action on $M_k'$ unless stated otherwise.\\

Now that we have the $T^n$-action, we will, as explained in \cref{rem:identifications}, take the isomorphisms $f_k$ with respect
to $T^n$-equivariant identifications of a neighborhood of $\mathcal{M}_k$ with $\mathcal{M}_k\tilde{\times} D^{2(n-k-1)}$.
\begin{remark}\label{chain}
	It should be noted that there are inclusions
    \begin{align*}
	X_1^j\subset X_2^j \subset \hdots \subset X_{j-1}^j
	\end{align*}
	by construction of the $X_k^j$, since $X_{k+1}$ is obtained from $X_k$ by removing a neighborhood of $\mathcal{M}_k$ and gluing in
	$F_{k+1}\tilde{\times} S^{2n-2k-3}=F_{k+1}\tilde{\times} \partial D^{2(n-k-1)}$	using $f_k$ as in the proof of \cref{defretr}. We may call this process 'equivariant surgery'
	due to the similarities to the process of surgery.\\
	
	Take a generic point in a connected component $C$ of $F_k$, $k\leq n-3$, and consider its stabilizer in $T$. This is a subtorus $T'\subset T$ of dimension $n-k-1$, acting
	on $D^{2(n-k)}$ linearly. Since $T'$ is contained in every (abelian) stabilizer $T_p$ of $F_k$,
	and every such stabilizer acts on $T^{n-k}\subset D^{2(n-k)}$, $T_p/T'$ acts on $T^{n-k}/T'=S^1$. It follows that the $T_p/T$-representation
	on $S^1$ factors through that of a cyclic group $G$.
	Since $D^{2(n-k)}/T'$ is homeomorphic to a disk of dimension $n-k+1$, so is $D^{2(n-k)}/T_p$. Indeed,
	we can think of the boundary of $D^{2(n-k)}/T'$, a sphere of dimension $n-k$, as $D^{n-k-1}\times S^1$ with the $S^1$-fibers collapsed
	in the boundary, and under this identification, the action of $G$ on said sphere is simply given by rotation in the fiber.\\
	While $(C\tilde{\times} D^{2(n-k)})/T'=C\tilde{\times} (D^{2(n-k)}/T')$ is
	a $D^{n-k+1}$-bundle over $C$, it is important to stress that this does not necessarily make
	$(C\tilde{\times} D^{2(n-k)})/T$ a $D^{n-k+1}$-bundle over $C^*$ as it is if $R=\Z$, or more generally, if the stabilizer on $C$ is constant.\\
	
	When we define $S':=S_{j-k}^{2n-2k-3}$ to be $S^{2n-2k-3}$ with all torus orbit (with respect to the maximal linear torus action on $S^{2n-2k-3}$ )
	of dimension less than $j-k$ removed and set $S:=(S_{j-k}^{2n-2k-3})^*$ (with respect to any linear $T^{n-k-2}$-action whose orbit space is a manifold;
	the homeomorphism type of the resulting orbit space does not depend on the choice of such action), $X_{k+1}^j$ is obtained from $X_k^j$
	(for $j\leq k+2$ of course) by gluing in $F_{k+1}\tilde{\times} S_{j-1}^{2n-2k-3}$. For $R=\Z$, we get a diagram coming out of the Mayer-Vietoris sequences
	\[
	\begin{tikzcd}
		\hdots \overset{g_*}{\to} H_*(F_{k+1}^*\times S) \oplus H_*((X_{k}^{j})^*) \arrow{r} \arrow{d} & H_*((X_{k+1}^{j})^*) \arrow{r} \arrow{d} & H_{*-1}(\mathcal{M}_{k}^*\times S) \arrow{d} \overset{g_{*-1}}{\to} \hdots\\
		\hdots \to H_*(F_{k+1}^*) \oplus H_*(M_{k}^*) \arrow{r} & H_*(M_{k+1}^*) \arrow{r} & H_{*-1}(\mathcal{M}_{k}^*) \to \hdots
	\end{tikzcd}
	\]
	We get a similar diagram for the total spaces. Later, we will explain how the maps $g_*$ look like (for orbit spaces and 'total spaces').\\
	The cases especially important to us are $j=k+2$ and $j=n$.
	The first case corresponds to the case of 'surgery' on $X_k^*$ itself
	(because $X_k^{k+2}=X_k\setminus \mathcal{M}_k$ and $X_{k+1}^{k+2}=X_{k+1}$), and the second one corresponds to the inclusions $Y_1\subset Y_2\subset \hdots$,
	and $Y_{k+1}$ is obtained by gluing $F_{k+1}\tilde{\times} (\mathring{D}^{n-k-2}\times T^{n-k-1})$ onto $Y_k$ via
	\[
	\partial F_{k+1}\tilde{\times} (\mathring{D}^{n-k-2}\times T^{n-k-1})\overset{f_k}{\rightarrow} \mathcal{M}_k\tilde{\times} (\mathring{D}^{n-k-2}\times T^{n-k-1})= Y_k\cap (\mathcal{M}_k\tilde{\times} S^{2n-2k-3}).
	\]
\end{remark}
Although the following statements hold similarly for $R=\Q$, the way to show them is way more involved than for $R=\Z$.
Since, in addition, we do not really need these statements for $R=\Q$ later, we will assume $R=\Z$ from now on until the end.
That is, we assume in particular that the normal bundle of $\mathcal{M}_k$ in $X_k$ is equivariantly trivial for $k\leq n-3$ (see \cref{equlinebundles1}
respectively \cref{equlinebundles2}).\\
Now let us look at the map $g_*=g^1_*\oplus g^2_*$ in the above diagram. As explained in \cref{rem:identifications}, we have fixed a $T^n$-equivariant
embedding of the normal bundle of $\mathcal{M}_k$ into $X_k$ and defined $f_k$ relative to this fixed embedding. So (on level of orbit spaces, but also for the total spaces)
$g^2_*$ is just induced by the embedding $\mathcal{M}_k^*\times S\to (X_k^j)^*$. However, $g^1_*$ is induced by the inclusion $\partial F_{k+1}^*\times S\to F_{k+1}^*\times S$
precomposed with $f_k^{-1}\colon \mathcal{M}_k^*\times S\to \partial F_{k+1}^*\times S$. Remember that $\mathcal{M}_k^*$ and $\partial F_{k+1}^*$
are the same sets (\cref{rem:boundaryidentification}). Thus, the formulation of the following lemma is justified.
\begin{lemma}\label{lem:homologymap}
	Assume $(k,j)\neq (1,n)$.
	On level of homology of the orbit spaces, $f_k$ is the identity.\\
	The same statement also holds for the 'total spaces' if, in addition, $(n,k)\neq (6,3)$.
\end{lemma}
\begin{proof}
	Since $f_k$ is given by
	\[
	f_k(p,v)=(p,h_k(p)(v))
	\]
	and $h_k$ takes values in a torus if $k\leq n-4$, the assertion follows for orbit spaces and total spaces because $h$ is nullhomotopic for $k\geq 2$
	and a $T^n$-orbit in $S_{j-k}^{2n-2k-3}$ and $(S_{j-k}^{2n-2k-3})^*$ is nullhomologous for $j\neq n$.\\
	
	For $k=n-3$, we use that $h_k$ is invariant under the $T$-action on $\mathcal{M}_k$, so that the effect of $f_k$ on homology is governed by
	the effect on homology of the induced maps
	$$\mathcal{M}_k^*\overset{h_k^*}{\to} \U(2)\to S^3 \text{  and  } \mathcal{M}_k^*\overset{h_k^*}{\to} \U(2)\to S^3\to S^2$$
	for total space and orbit space, respectively.\\
	This is clearly always trivial for the orbit spaces. For the total spaces it is trivial for $n\neq 6$ by degree reasons.
	This shows the lemma.
\end{proof}
\begin{remark}\label{rem:homologyhtilde}
	In the case of $(n,k)=(6,3)$ it could certainly happen that $f_k$ is not the identity because the induced map $\mathcal{M}_k^*\overset{h_k^*}{\to} \U(2)\to S^3$
	does not have to be trivial. The non-trivial effect of $f_k$ on homology,
	\[
	(H_3(\mathcal{M}_3^*)\oplus (H_0(\mathcal{M}_3^*)\otimes H_3(S^3)))\otimes H_{*}(T^4)\to (H_3(\mathcal{M}_3^*)\oplus (H_0(\mathcal{M}_3^*)\otimes H_3(S^3)))\otimes H_{*}(T^4),
	\]
	can then be described by $(x,y)\otimes z\mapsto (x,\tilde{h}(x)+y)\otimes z$, where $\tilde{h}\colon H_*(\mathcal{M}_3)\to (H_0(\mathcal{M}_3)\otimes H_3(S^3))$
	is induced by $h_3^*$.
\end{remark}
\begin{lemma}\label{section}
	In the situation of \cref{fol} and $k\leq n-3$, the restriction of $r_k$ to $B_k:=Y_k/T^n$ induces an isomorphism in homology between $B_k$ and $M_k^*$.\\
	This also holds for $k=n-2$ when the $T^n$-action is defined on $X_{n-2}$.
\end{lemma}
\begin{proof}
	This is true for $k=0$, because then $Y_0$ is the disjoint union of $\mathring{D}^{n-1}\times T^n$, one for each vertex of $\Gamma$, and $M_0$ is the vertex set
	of $\Gamma$. So assume it holds for some $k$.
	Since $B_{k+1}\setminus B_k\cong F_k^*\times \mathring{D}^{n-k-2}$ is acyclic, we may consider the diagram for $*\geq 1$ (the horizontal sequences are
	Mayer-Vietoris sequences)
	\[
	\begin{tikzcd}
		\hdots \arrow{r} & H_*(B_k) \arrow{r} \arrow{d} & H_*(B_{k+1}) \arrow{r} \arrow{d} & H_{*-1}((\mathcal{M}_k\tilde{\times} S')/T^n) \arrow{d} \arrow{r} & \hdots\\
		\hdots \arrow{r} & H_*((M_{k+1}\cap M_k')^*) \arrow{r} & H_*(M_{k+1}^*) \arrow{r} & H_{*-1}(\mathcal{M}_k^*) \arrow{r} & \hdots
	\end{tikzcd}
	\]
	where the vertical maps are induced by $r_{k+1}^*$. Now, the left vertical map is the same as the one induced by $r_k$
	(since $r_k$ and $r_{k+1}$ are equivariantly homotopic on $Y_k$) and the assertion follows from the inductivity assumption and the five lemma.
\end{proof}
\begin{remark}
	It is certainly true that $B_k$ is homotopy equivalent to $M_k^*$ via $r_k$ for $R=\Z$, but we do not need this.
\end{remark}
\begin{lemma}\label{lem:homologyfkj}
	A connected component $C$ of $(\mathcal{F}_k^j)^*$, where $k\leq n-3$ and $k+1\leq j\leq n-2$, is $k-1$-acyclic.
\end{lemma}
\begin{proof}
	Note that $C$ is given by $X_k^j\cap Z$ for a certain face submanifold $Z\subset M$ (namely that to which $C$ is mapped under $r_k$) of dimension $2j$. By \cref{cor:defretrnaturality} combined with
	\cref{section}, $C^*=(X_k^j\cap Z)^*$ has the same homology as $(M_k\cap Z)^*$, that is, the same homology as the orbit space
	of the equivariant $k$-skeleton of an equivariantly formal torus manifold. This shows the claim.
\end{proof}
\begin{remark}\label{equlinebundles3}
	For $n\geq 6$ it is immediate that $Y_k\cong B_k\times T^n$ for $k\leq n-3$, since $H^2(B_k)=0$ for $k\neq 2$
	and $B_2\subset B_3$.\\	
	For $n=5$ this holds nonetheless. The reason is that $H_2(X_2^*)$ is generated by $H_2(\mathcal{M}_2^*)$, and then
	the bundle $T^n\to Y_2\to B_2$ is trivial. Indeed, the normal bundle of $\mathcal{M}_2$ is a sum of line bundles
	as seen in \cref{equlinebundles2}, and this extends to a $\C^2$-bundle over $F_3$ whose first Chern class is $0$.\\
	It follows that the $T^n$-bundle restricted to the intersection of $B_2$ with a small neighborhood of $\mathcal{M}_2^*$ in $X_2^*$ is trivial.
	Since $H_2(B_2)$ is generated by the intersection of $B_2$ with a small neighborhood of $\mathcal{M}_2^*$ in $X_2^*$,
	the claim follows.
\end{remark}
\begin{lemma}
	The normal bundle of $\mathcal{F}_k^j$, where $k\leq n-3$, $k+1\leq j$, in $X_k^j$ is equivariantly trivial. The same holds for $T$-orbit spaces.
\end{lemma}
\begin{proof}
	For $j\leq n-2$, said normal bundle is a restriction of the normal bundle of $\mathcal{M}_{j-1}$ in $X_{j-1}$, which is equivariantly trivial
	(the same holds for $T$-orbit spaces). We need to check $j=n-1$.\\
	The boundary of the normal bundle of $\mathcal{F}_k^{n-1}$ in $X_k^{n-1}$ is a principal $S^1$-bundle. Its total space carries the natural structure
	of a principal $T^n$-bundle (remember that we have a $T^n$-action on $X_k^{n-1}$), extending the $S^1$-action on the fiber. The embedding
	of this $T^n$-bundle into $Y_k$ is a morphism of principal $T^n$-fiber bundles.
	Since $T^n\to Y_k\to B_k$ is a trivial bundle by \cref{equlinebundles3}, we get the claim.  
\end{proof}
We defined $X_k^j$ by removing certain closed subsets out of $X_k$. However, the way
we obtained $X_{k+1}$ from $X_k$ was by removing a small open neighborhood of $\mathcal{M}_k$,
to obtain a manifold with boundary we now call $X_k'$, and gluing in something into $X_k'$ along
its boundary. This makes it so that $\mathcal{F}_k^{k+2}\subset X_k'$ can be considered as a manifold
with boundary. We will make use of this in the next theorem.
\begin{theorem}\label{homologyX*}
	Assume that the homology of $M_k^*$ is nontrivial only in degrees $0$ and $k$ for $k\leq n-2$.
	Then the map $H_*(X_k^*)\to H_*(M_k^*)$ is surjective and its kernel is $H_n(X_k^*)$ plus the image of
	$H_0(\mathcal{M}_{k-1}) \otimes H_*(S^{n-k}) \to H_*((X_{k-1}')^*)\to H_*(X_k^*)$.\\
	If the $T^n$-action exists on whole $M$, then the same is true for $k=n-1$.
\end{theorem}
\begin{proof}
	The surjectivity of $H_*(X_k^*)\to H_*(M_k^*)$ for $k\leq n-3$ follows from \cref{section}, because the $S^1$-bundle is trivial due to \cref{equlinebundles3}.
	For $k=n-2$ it is true as well because of the following diagram for $k=n-3$ ($*$ is supposed to be $\geq 1$)
	\[
	\begin{tikzcd}
		\hdots \to [H_0(\mathcal{M}_k)\otimes H_*(S^{n-k-1})] \oplus H_*((X_k^{k+2})^*) \arrow{r} \arrow{d} & H_*(X_{k+1}^*) \arrow{r} \arrow{d} & H_{*-1}(\mathcal{M}_k^*\times S^{n-k-1}) \arrow{d} \overset{g}{\to} \hdots \\
		\hdots \to 0 \oplus H_*(M_k^*) \arrow{r} & H_*(M_{k+1}^*) \arrow{r} & H_{*-1}(\mathcal{M}_k^*) \overset{h}{\to} \hdots
	\end{tikzcd}
	\]
	We know that $H_{*}(\mathcal{M}_k^*\times \{pt.\})\to H_{*}((X_k^{k+2})^*)$ factorizes over $H_{*}(Y_k^*)$ and that $H_{*}(\mathcal{M}_k^*\times \{pt.\})$
	is sent to $0$ in $H_0(\mathcal{M}_k)\otimes H_*(S^{n-k-1})$ by \cref{lem:homologymap}. This implies that the kernel of $h$ is the image of the kernel of $g$.\\
	
	The other statements are already shown for $k=1$ (see \cref{X_1*}), so we assume that they hold for some $k\leq n-3$.
    Since the isotropy action of $T$ on a sphere in the fiber of the normal bundle of $\mathcal{M}_k$
	is in general position by \cref{rem:isotropygenpos}, the orbit space of this sphere is a sphere of dimension $n-k-1$, and the process of 'equivariant surgery' (\cref{chain}) on $X_k$ corresponds to
	the process of surgery on $X_k^*$. Moreover, all orbit spaces $X_k^*$ are topological manifolds by \cref{rem:orbitspacemanifold}, which enables us
	to use Poincaré duality throughout, that is, $b_i(X_k^*)=b_{n-i}(X_k^*)$ and $\text{Tor}(H_i(X_k^*))=\text{Tor}(H_{n-i-1}(X_k^*))$ (using the universal coefficient theorem for cohomology in the last statement, of course).\\
	
	Assuming at first that $k< \lfloor \frac{n}{2} \rfloor$, $H_k((X_k')^*)\to H_k(X_{k}^*)$ is an isomorphism, and so $H_k((X_k')^*)\to H_k(M_k^*)$
	is. Hence, using the above diagram, $H_k(X_{k+1}^*)=H_k(M_{k+1}^*)=0$, which implies $H_{n-k}(X_{k+1}^*)=0$ using Poincaré duality, and
	$H_{k+1}(X_{k+1}^*)\to H_{k+1}(M_{k+1}^*)$ is injective, thus an isomorphism. All in all, $H_*(X_{k+1}^*)$ is concentrated in the desired degrees
	and the kernel of $H_*(X_{k+1}^*)\to H_{*}(M_{k+1}^*)$ is generated by $H_n(X_{k+1}^*)$ plus the image of
	\[
	H_0(\mathcal{M}_{k-1}) \otimes H_*(S^{n-k-1}) \to H_*((X_{k}')^*)\to H_*(X_{k+1}^*)
	\]
	as claimed.\\
	If $k=\lfloor \frac{n}{2} \rfloor$, we need to argue differently. Assume that $n\geq 6$ is even (there is nothing to do for $n=4$). We still want to show that $H_k(X_{k+1}^*)=0$.
	We have $H_k(Y_{k+1}^*)$ because of \cref{section}. Now we consider the chain of inclusions
	\[
	Y_{k+1}^*=(X_{k+1}^n)^*\to (X_{k+1}^{n-1})^*\to (X_{k+1}^{n-2})^*\to \hdots \to (X_{k+1}^{k+2})^*=X_{k+1}^*
	\]
	Using that $(\mathcal{F}_{k}^j)^*$ is $k-1$-acyclic for $j\leq n-2$ (see \cref{lem:homologyfkj}), the Mayer Vietoris sequence (where $m-1\geq k+2$)
	\[
	\hdots \to H_*((\mathcal{F}_{k+1}^{m-1})^*)\oplus H_*((X_{k+1}^{m})^*)\to H_*((X_{k+1}^{m-1})^*)\to H_{*-1}((\mathcal{F}_{k+1}^{m-1})^*\times S^{n-m+1})\to \hdots
	\]
	shows that the cokernel of $H_k((X_{k+1}^{m})^*)\to H_k((X_{k+1}^{m-1})^*)$ is trivial except for $m=n$, where it could happen that $H_{k-2}((\mathcal{F}_{k+1}^{n-1})^*)\neq 0$.	
	So, in our current situation, we have that the kernel of $H_*(X_{k+1}^*)\to H_*(M_{k+1}^*)$ is given by
    $$H_{k-1}(X_{k+1}^*)\oplus H_{k}(X_{k+1}^*)\oplus H_{n}(X_{k+1}^*),$$
	so this is an isomorphism in degree $k+1$.\\
	
	Now assume $n\geq 8$. If we start at $k=n/2$ and go up with $k$ just as before, doing the surgery on $X_k^*$ in every step, we see that
	$H_{j}(X_{n-2}^*)\to H_j(M_{n-2}^*)$ is an isomorphism for $k+1\leq j\leq n-1$.
	This, in turn, implies that $H_{\leq k-2}(X_{n-2}^*)=0$ unless in degrees $0$ and $2$, and that
	$H_j(X_{n-3}^*)\to H_j(M_{n-3}^*)$ is an isomorphism as well. The MVs
	\[
	\hdots \to H_*((\mathcal{F}_{n-3}^{n-1})^*)\oplus H_*(Y_{n-3}^*)\to H_*((X_{n-3}')^*)\to H_{*-1}((\mathcal{F}_{n-3}^{n-1})^*\times S^1)\to \hdots
	\]
	gives us (using also that $H_{n-3}(Y_{n-3}^*)\to H_{n-3}((X_{n-3}')^*)\to H_{n-3}(M_{n-3}^*)$ is an isomorphism)
	$$H_{k-1}((\mathcal{F}_{n-3}^{n-1})^*)=\hdots =H_{n-5}((\mathcal{F}_{n-3}^{n-1})^*)=H_{n-4}((\mathcal{F}_{n-3}^{n-1})^*)=0.$$
	That is, we have $H_*((\mathcal{F}_{n-3}^{n-1})^*)=0$ for $k-1=\frac{n-2}{2}\leq *\leq n-4$. But $(\mathcal{F}_{n-3}^{n-1})^* \subset (X_{n-3}')^*$ is a compact manifold
	of dimension $n-2$ whose boundary is a union of homology spheres, so $H_*((\mathcal{F}_{n-3}^{n-1})^*)=0$ for all $2\leq *\leq n-4$ by Lefschetz duality.
	A posteriori, we find that
	$$H_{k-2}((\mathcal{F}_{k+1}^{n-1})^*)=H_{k-2}((\mathcal{F}_{n-2}^{n-1})^*)=0$$, thus $H_k(X_{k+1}^*)=0$, and we are done.\\
	We then also have $H_1((\mathcal{F}_{n-3}^{n-1})^*)=0$, for if not, then $H_3((X_{n-3}')^*)\neq 0$ would inject into $H_3(X_{n-2}^*)=0$.
	That is, the homology of $(\mathcal{F}_{n-3}^{n-1})^*$ is generated by its boundary.\\
	
	Now, for the case $n=6$, if $H_1((\mathcal{F}_3^{n-1})^*)\neq 0$, then $b_1((\mathcal{F}_3^{n-1})^*)\neq 0$
	by \cref{rem:graphbetti}, and then $H_3((\mathcal{F}_3^{n-1})^*)$ is not generated by its boundary
	(using Lefschetz duality),
	so $H_5((X_3^5)^*)$ is not generated by its boundary $\mathcal{M}_3^*\times S^2$ (because the homology of $\mathcal{M}_3^*\times S^2$
	in $H_5((X_3^5)^*)$ is mapped to $H_3((\partial \mathcal{F}_3^{5})^*)\otimes H_1(S^1)$ under the boundary map of the MVs above). This contradicts $H_5(X_4^*)=0$.\\
	
	Let us now turn to odd $n\geq 7$.
	Here, we want to show that $H_{k+1}(X_{k+1}^*)\to H_{k+1}(M_{k+1}^*)$ is an isomorphism. Certainly $H_{k+1}(Y_{k+1}^*)\to H_{k+1}(M_{k+1}^*)$
	is, so by the same reasons as for even $n$, we need to argue why $H_{k-1}((\mathcal{F}_{k+1}^{n-1})^*)=0$.\\
	Since $(\mathcal{F}_{n-3}^{n-1})^*$
	is obtained from $(\mathcal{F}_{k+1}^{n-1})^*$ by attaching homological cells of degree at least $k+2$, the inclusion induces an isomorphism in
	homology of degree $\leq k$, and then we have with Lefschetz duality
	$$b_{k-1}((\mathcal{F}_{k+1}^{n-1})^*)=b_{k-1}((\mathcal{F}_{n-3}^{n-1})^*)\overset{\text{LD}}{=}b_{k}((\mathcal{F}_{n-3}^{n-1})^*)=b_{k}((\mathcal{F}_{k+1}^{n-1})^*)$$
	and
	$$b_{k}((\mathcal{F}_{k+1}^{n-1})^*)=b_{k+2}((X_{k+1}^{n-1})^*)=b_{k+2}((X_{k+1}^{n-2})^*)=\hdots = b_{k+2}(X_{k+1}^*)=0,$$
	where the last equations come from the obvious Mayer Vietoris sequence. Therefore, we could only have torsion. But this would imply
	that $H_{k+1}((X_{k+1}^{n-1})^*)$ has torsion (the short exact sequence
	\[
	0 \to H_{k+1}(Y_{k+1}^*)\to H_{k+1}((X_{k+1}^{n-1})^*)\to H_{k-1}((\mathcal{F}_{k+1}^{n-1})^*)\otimes H_1(S^1)\to 0
	\]
	splits indeed, because $H_{k+1}((X_{k+1}^{n-1})^*)\to H_{k+1}(M_{k+1}^*)\overset{\cong}{\to} H_{k+1}(Y_{k+1}^*)$ is a split)
	which injects into $H_{k+1}(X_{k+1}^*)$. But this is impossible due to Poincaré duality and $H_{k-1}(X_{k+1}^*)=0$.\\
	When we now go up, we see that the kernel of $H_*((X_{n-3}^{n-1})^*)\to H_*(M_{n-3}^*)$ is non-trivial at most in degrees $2$ and $k$
	with the same arguments used as in the case $k<\lfloor \frac{n}{2}\rfloor$.	In $k$, it could only be a torsion group $T$ by PD. By the same MVs,
	we would have $H_{k-2}((\mathcal{F}_{n-3}^{n-1})^*)=H_k((\mathcal{F}_{n-3}^{n-1})^*)=T$, but then there would be non-trivial kernel of
	$H_*((X_{n-3}^{n-1})^*)\to H_*(M_{n-3}^*)$ in degree $k+2$ with the same arguments as before.\\
	
	For $n=5$, we only need to show that $H_3((X_2^4)^*)=0$, which is equivalent to $H_1((\mathcal{F}_2^4)^*)=0$, which is equivalent to $b_1((\mathcal{F}_2^4)^*)=0$
	by \cref{rem:graphbetti}. This can be argued as for $n=6$.
\end{proof}
As a side product, we have proven
\begin{lemma}\label{homologyFkn-1}
	Let $n\geq 5$ and $F_k$ be a connected component of $\mathcal{F}_k^{n-1}$ (for all $k\leq n-2$ when $F_k$ belongs to an isotropy submanifold, and $k\leq n-3$
	otherwise). The homology of $F_k^*$ is concentrated in degree $k$ and, if it is defined, $F_{n-2}^*$ has the homology of a sphere.
\end{lemma}
For $n=4$ and $R=\Z$, the situation is not so clear. It is not even clear if isotropy submanifolds of $X_2$ are automatically orientable. To motivate this,
consider the spaces $M_1'$ and $X_1=\partial M_1'$ associated to $\Gamma$ as in \cref{X_1^*}. These are automatically orientable by \cref{orientable},
because $\pi_1(X_1)$ is generated by $\mathcal{M}_1$. We need to think about the singular foliation $\mathcal{F}_1$ on $X_1$
as in \cref{singfol}, particularly in a neighborhood of a connected component $C$ of $\mathcal{M}_1$, depending on whether the connection path $\gamma_C$ associated to $C$ has trivial
monodromy or not.\\
If it does, then $\mathcal{F}_1^3$, the space of leaves of $\mathcal{F}_1$ dimension $3$, intersected with $C\times S^3$, the boundary of a neighborhood of $C$ in $X_1$, is simply given
by the two components $C\times (S^1\times \{0\})$ and $C\times (\{0\}\times S^1)$ (as it is the case when the labeling of $\Gamma$ extends to a $\Z^4$-labeling), whereas
if the monodromy of $\gamma_C$ is not trivial, then $\mathcal{F}_1^3\cap (C\times S^3)$ will only have one connected component. The reason is that, when thinking of
$C\times S^3$ as $[0,1]\times T^2\times S^3$ glued together along its ends, $\mathcal{F}_1^3$ on there is given by the union of $[0,1]\times T^2\times (S^1\times \{0\})$
and $[0,1]\times T^2\times (\{0\}\times S^1)$ in both cases, but the gluing map in case of a trivial monodromy is the identity, and the map for
non-trivial mondromy on $S^3$ is given by switching the entries. By the same reasoning, the space $Y_1^*=(\mathcal{F}_1^4)^*$ might not be the total space of a trivial
$S^1$-bundle over $B_1$ anymore (compare with \cref{section}), but only the total space of an unorientable $S^1$-bundle over $B_1$ (which implies that $B_1$ itself is necessarily
unorientable, then).\\

Albeit these difficulties in a neighborhood of $M_1$ constructed out of a graph $\Gamma$ occur, we will show the following.
\begin{lemma}\label{homologyFkn-1dim8}
	If $M_2'$ can be embedded in a $\Z$-equivariantly formal $T^3$-manifold $M$ (which is then necessarily orientable),
	then any isotropy submanifold $Z$ of $X_2$ is indeed orientable and its orbit space is a two-sphere. If, moreover, every connected isotropy submanifold of $M$ contains a fixed point,
	then $T\to Z\to Z^*$ is a trivial bundle.
\end{lemma}
\begin{proof}
Let $C$ be the connected component of isotropy submanifolds of $M$ that contains $Z$.
After dividing out the kernel of the $T$-action on $C$, we may view the $T$-action to be effective. Let $C'\subset C$ be a connected component of
those isotropy submanifolds whose principal stabilizer $H$ is disconnected. A neighborhood in $C$ of a $T$-orbit through a point in $C'$ is equivariantly diffeomorphic
to $T\times_H \R^3$, where $H$ acts as a subgroup of $\SO(3)$ on $\R^3$. Since $H$ is abelian, it is necessarily cyclic by the classification of finite subgroups of $\SO(3)$,
and thus the orbit space of $C'$ has dimension one. Thus, the only stabilizer subgroup occuring in $C'$ is $H$, and $C'/T$ is $S^1$. Moreover, since the orbit space of $T\times_H \R^3$
is the orbit space of $\R^3/H$ and thus homeomorphic to a sphere, $C^*$ is a topological manifold of dimension $3$ with boundary $\partial C^*\cong Z^*$.\\
Let us show that $C$ is orientable. If $H$, generated by some element $h$, is not of order two, the sphere bundle belonging to the normal bundle of $C$ in $M$ can be oriented by saying that
a positive orientation of a fiber is given by the direction of a geodesic from some point $p$ to $h\cdot p$ not passing through $h^2\cdot p$. If $H$ is of order two, then,
since $M$ is equivariantly formal over $\Z$, we have
\[
\dim H^{odd}(C;\Z_2)\leq \dim H^{odd}(M;\Z_2)=0,
\]
so in particular $w_1(C)\in H^1(C;\Z_2)$ is the $0$-class, which makes $C$ orientable again.\\
In order to show that $Z^*$ is a two-sphere, we only need to show that $Z^*$ is a rational homology sphere, or equivalently, that $C^*$ is a rational homology disk.
By the above inequality with coefficients in $\Z_p\subset H$, $p$ a prime, we see that $C$ has vanishing odd rational cohomology, so that $C^*$ is necessarily rationally
face-acyclic by \cref{torusManifoldgen}.\\
As a consequence, if every connected isotropy submanifold of $M$ contains a fixed point, the first Chern class of the principal bundle $T\to Z\to Z^*$ is the image of
the first Chern class of the principal bundle $T\to C\to C^*$ under $H^2(C^*)\to H^2(Z^*)$. But $H^2(C^*)$ is torsion, so the bundle $T\to Z\to Z^*$ is trivial as claimed.
\end{proof}

\newpage

\section{Equivariant formality and orbit spaces of manifolds in general position}\label{EFgenPos}
From now on and until the end, we take coefficients to be $\Z$ unless stated otherwise.
In this section, we want to link the equivariant formality of a manifold in general position (with certain isotropy assumptions)
to the homology of its orbit space. The next lemma more or less states that, although $M$ does not have to be an equivariantly formal torus manifold,
$X:=X_{n-2}$ has the 'homological properties' as if $M$ was an equivariantly formal torus manifold.
The exclusion of $n=6$ in the last point seems a bit peculiar at first, but this will turn out to be fine later, when we look at the greater context.
\begin{lemma}\label{propX}
	Let $X=X_{n-2}=\partial M_{n-2}'$ be as in \cref{local} respectively \cref{defretr}. If $n=4$, assume in addition that $M$ is compact, equivariantly formal
	over $\Z$ and that every component of an isotropy submanifold contains a fixed point. Denote by $\partial Z$ (this could be empty) the isotropy submanifolds of $X$
	(by \cref{homologyFkn-1} and \cref{homologyFkn-1dim8}, these are products of homology spheres and tori).
	\begin{enumerate}
		\item We have an isomorphism $H_2(X)\to H_2(M_{n-2})$ for $n\geq 5$.\\
		
		\item Let $n\neq 6$. The following hold:\\
		\begin{enumerate}
		\item For even $4\leq j\leq n$, the kernel of $H_j(X)\to H_j(M_{n-2})$ is free abelian and equal to the image of
		\[
		H_j(X_1^{n-1})\to H_j(X).
		\]
		\item For odd $j\leq n$, $H_{j}(X)\to H_{j}(M_{n-2})$ is surjective.\\
		
		\item For even $4\leq j\leq n-2$, the cokernel of $H_j(X)\to H_j(M_{n-2})$ is free abelian.\\
	    \end{enumerate}
		If $d^2_{4,*}=d^3_{4,*}=d^4_{4,*}=0$ for the rational Serre spectral sequence associated to $T\to X\to X^*$, then this holds for $n=6$ as well.\\
		
		\item For $n\neq 6$, $H_k(X)$ is free abelian for all $k\neq n-1$. If $d^2_{4,*}=d^3_{4,*}=d^4_{4,*}=0$ for the rational Serre spectral sequence
		associated to $T\to X\to X^*$, then this holds for $n=6$ as well.\\
		
		\item For $n\geq 5$, $H_1((X\setminus \partial Z)^*)\neq 0$ if and only if every subgraph of covalence one is ineffective,
		and then $H_1((X\setminus \partial Z)^*)$ is generated by the $S^1$-fiber of $S^1\to Y_{n-2}^*\to M_{n-2}^*$.\\
		The same holds for $n=4$ if any three edges emerging at a vertex belong to an ineffective subgraph.\\
		
		\item $H_*((X\setminus \partial Z)^*)$ is torsion free.
	\end{enumerate}
\end{lemma}
To prove this, we need the following general principle.
\begin{lemma}\label{conttorus}
	For $n\geq 5$ and $X$ as above, there is an equivariantly formal torus manifold $\tilde{M}$ (which is not necessarily smooth, but smooth in a neighborhood of $\tilde{M}_{n-2}$)
	such that $M_{n-3}'= \tilde{M}'_{n-3}$. In particular, the GKM graph of $M$ is that of $\tilde{M}$, and so $H_T^*(M)\cong H_T^*(\tilde{M})$.
\end{lemma}
\begin{proof}
	For $n\geq 6$, we take $M_{n-3}'$ and glue in $F_{n-2}\times D^4$ into there in such a way that the
     map restricted to $\partial F_{n-2}\times \{0\}$ is the same as for $M_{n-2}$, and that the $T^n$-action is preserved. This gives us $\tilde{M}_{n-2}'$ with boundary
     $\tilde{X}_{n-2}$. Now $\tilde{\mathcal{M}}_{n-2}$ is equivariantly diffeomorphic to products of homology spheres and $T^{n-1}$ by \cref{homologyFkn-1}.
     Since any smooth homology sphere of dimension at least $4$ bounds a contractible manifold (though not necessarily smoothly) by \cite{Ker69}, we can perform equivariant surgery on $\tilde{\mathcal{M}}_{n-2}$
     using those contractible manifolds, giving us a manifold $\tilde{M}_{n-1}'$ with boundary $\tilde{X}_{n-1}$, which is the total space
     of a principal $T^n$-bundle over a closed manifold of dimension $n-1$ whose homology is concentrated in degrees $0$, $n-2$ and $n-1$, hence a homology sphere of dimension $\geq 5$
     (actually this is a homotopy sphere, since the fundamental group of $\tilde{X}_{n-2}^*$ is generated by $\tilde{\mathcal{M}}_{n-2}$, but this is not important).
     Since any homology sphere of dimension $\geq 5$ is smooth (they have a PL structure by Kirby-Siebenmann, and then it was observed in \cite{Ker69} that they have a smooth structure),
     it bounds a contractible manifold topologically by \cite{Ker69}, so we can close $\tilde{M}'_{n-1}$ equivariantly to obtain a
     $T^n$-manifold $\tilde{M}$ whose orbit space is face-acyclic.\\
     For $n=5$ we do this in the same way, but instead of gluing $F_3\times D^4$ into $M_2'$, we take $D^3\times T^3\times D^4$, ensuring that $\tilde{\mathcal{M}}_{n-2}^*$
     is a union of $3$-spheres, thus ensuring that the boundary of $\tilde{M}'_{4}$ is actually a smooth $4$-sphere (it is unknown to the author whether every homology $4$-sphere is smoothable).
\end{proof}
\begin{remark}
	This fits into the context of the statement in \cite{AM23} that the equivariant cohomology of manifolds in general position for $n\geq 5$
	with $\Q$-coefficients has the structure of the equivariant cohomology of an equivariantly formal torus manifold with the $T^n$-action restricted to
	a certain $T^{n-1}$ in $T^n$.\\
	Note that it does not really matter that $\tilde{M}$ constructed above is not smooth, since the smoothness used in \cite{MP03} was only used to ensure the existence
	of local models around $T^n$-orbits, the existence of normal bundles of face submanifolds and then the existence of a canonical local model of the torus manifold,
	all of which we still have in our case. A bit of care might be necessary in section 9 of the article \cite{MP03}, since blow-ups are considered there. However, there is no blow-up performed
	on a characteristic submanifold (this wouldn't even change the space), so that every blow-up occuring in these arguments is performed on facial submanifolds of codimension at least
	$4$, all of which are contained in the 'smooth part' $\tilde{M}_{n-2}'$ of $\tilde{M}$.
\end{remark}
\begin{proof}[Proof of \cref{propX}]
	\mbox{}
	\begin{enumerate}
		\item We have an isomorphism $H_2(X_1)\to H_2(M_1)$ by \cref{X_1^*}. Since the submanifolds we iteratively remove from $X_1$ to obtain $X_1^{n-1}$
		are of codimension at least $4$, we do not change the second homology when doing so. That is, $H_2(X_1^{n-1})\to H_2(M_1)$ is also an isomorphism.
		Now consider the diagram in \cref{chain} for $j=n-1$ and iterate to see that $H_2(X_k^{n-1})\to H_2(M_k)$ is an isomorphism for all $k\leq n-2$, in particular $k=n-2$.\\
		
		\item Assume that $n\geq 5$. We show the assertion for an equivariantly formal torus manifold $M$ first and set $N:=M\setminus (M_{n-2}'\setminus X)$.
		Looking at the long exact sequence 
		\[
		\hdots \to H_*(X) \to H_*(M_{n-2})\oplus H_*(N)\to H_*(M)\to H_{*-1}(X)\to H_{*-1}(M_{n-2})\oplus H_{*-1}(N)\to \hdots
		\]
		for odd degrees, we see that $H_{odd}(X)\to H_{odd}(M_{n-2})\oplus H_{odd}(N)$ is surjective (which already shows statement (b), then),
		$H_{even}(X)\to H_{even}(M_{n-2})\oplus H_{even}(N)$ is injective and that its cokernel is free abelian (if there is torsion in even degree, then also in odd
		degree by Poincaré duality and the universal coefficient theorem).\\
		The orbit space of both the free and non-free $T^n$-orbits $N_{n-1}$ and $N_n$ of $N$ is acyclic. By the diagram (where $k<m$, of course)
		\begin{equation*}\label{diagram}
				\begin{tikzcd}
				\hdots \to H_j(\mathcal{F}_{k}^{n-1})\oplus H_j(Y_k) \arrow["h_1",r] \arrow{d} & H_j(X_k^{n-1}) \arrow{r} \arrow{d} & H_{j-1}((\mathcal{F}_{k}^{n-1})^*\times T^n) \arrow{d} \overset{h_2}{\to} \hdots\\
				\hdots \to H_j(\mathcal{F}_{m}^{n-1})\oplus H_j(Y_{m}) \arrow["g_1",r] \arrow{d} & H_j(X_m^{n-1}) \arrow{r} \arrow{d} & H_{j-1}((\mathcal{F}_{m}^{n-1})^*\times T^n) \arrow{d} \overset{g_2}{\to} \hdots\\
				\hdots \to H_j(N_{n-1})\oplus H_j(N_{n}) \arrow{r} & H_j(N) \arrow{r} & H_0(N_{j-1})\otimes H_{j-1}(T^n) \to \hdots
			\end{tikzcd}
		\end{equation*}
		the kernel of $H_*(X=X_{n-2}^{n-1})\to H_*(N)$ can be identified to be the image $I$ of $g_1$ plus the kernel $K$ of $g_2$ restricted to $H_{n-2}((\mathcal{F}_{n-2}^{n-1})^*)\otimes H_{*}(T^n)$.
		Thus we have the splitting $H_*(X)=I\oplus K\oplus K'$, where $K'$ is isomorphic to the kernel of $g_2$ restricted to $H_0((\mathcal{F}_{n-2}^{n-1})^*)\otimes H_*(T^n)$,
		which is isomorphic to the kernel of $h_2$ restricted to $H_0((\mathcal{F}_{k}^{n-1})^*)\otimes H_*(T^n)$ via the vertical map. Now for $k=1$, we choose a unique preimage
		(also called $K'$) in $H_*(X_1^{n-1})$ for this, and denote also by $K'\subset H_*(X_k^{n-1})$ the image of $K'\subset H_*(X_1^{n-1})$ under the inclusion $X_1^{n-1}\to X_k^{n-1}$.\\

			It follows that the image of $H_*(X_1^{n-1})\to H_*(X_k^{n-1})$ is precisely $K'$ for $k\geq 3$. Indeed, by the diagram, the image of $H_*(X_1^{n-1})\to H_*(X_2^{n-1})$
			is contained in $K'\oplus \text{im}(H_*(Y_2)\to H_*(X_2^{n-1}))$, and $H_*(Y_2)\to H_*(X_k^{n-1})$ is the $0$-map.\\
			For $j\geq 3$, we have that $H_j(X_1^{n-1})\to H_j(M_1)=0$ is the $0$-map. So in the splitting $H_*(X)=I\oplus K\oplus K'$, we may assume that
			the kernel of $H_j(X)\to H_j(M_{n-2})$ contains whole $K'$ and is thus equal to $K'$, because it has to intersect $I\oplus K$
			trivially (remember that $H_j(X)\to H_j(M_{n-2})\oplus H_j(N)$ is injective). This shows assertion (a) for torus manifolds.\\
			
		The only thing left to show for a torus manifold is that the cokernel of $H_j(X)\to H_j(M_{n-2})$ is free abelian for even $4\leq j\leq n-2$. We know this for $H_j(X)\to H_j(M_{n-2})\oplus H_j(N)$,
		and since $H_j(X)\to H_j(N)$ is surjective and the preimages may be chosen to be in $K'$, we can conclude that the cokernels of both maps agree.\\
		
		Now for general $M$ and $n\geq 5$, we use \cref{conttorus} and consider $\tilde{X}=\partial \tilde{M}_{n-2}$. Let $Q$ be either $\tilde{X}$ or $X$. Then we have the commutative diagram
		\begin{equation*}
			\begin{tikzcd}
				\hdots \overset{g_{j}}{\to} H_j(X_{n-3}')\oplus H_j(F_{n-2}\times S^3) \arrow{r} \arrow{d} & H_j(Q) \arrow{r} \arrow{d} & H_{j-1}(\mathcal{M}_{n-3}\times S^3) \arrow{d} \overset{g_{j-1}}{\to} \hdots\\
				\hdots \overset{h_{j}}{\to} H_j(M_{n-3})\oplus H_j(F_{n-2}) \arrow{r} & H_j(M_{n-2}) \arrow{r} & H_{j-1}(\mathcal{M}_{n-3}) \overset{h_{j-1}}{\to} \hdots
			\end{tikzcd}
		\end{equation*}
		and the maps $h_j$ are independent from $Q$. We know that $\ker(g_{j})\to \ker(h_j)$ is surjective for $\tilde{X}$ and even $j\leq n-1$ since $H_{j+1}(\tilde{X})\to H_{j+1}(M_{n-2})$ is surjective.
		So we know that the same holds for $X$ as long as $n\neq 6$, since then the maps $g_j$ are also independent from $Q$ by \cref{lem:homologymap} and the discussion before. As we will show in \cref{rem:cokernel}
		very soon, the cokernel of $\ker(g_j)\to \ker(h_j)$ for $n=6$ also does not depend on $Q,$ for $j\leq 5$ and under our assumptions.\\
        Let us now show the statements for $X$.\\
		\begin{enumerate}
		\item 
		By the already shown statement for $\tilde{X}$, the kernel of $H_j(\tilde{X})\to H_j(M_{n-2})$, where $4\leq j\leq n$ is even, is precisely $K'$.
		Using this and the above diagram, we see that $H_j(\mathcal{M}_{n-3}\times S^3) \oplus K'$ surjects onto
		\[
		\ker(H_j(X_{n-3}')\to H_j(M_{n-3})/\text{im}(H_j(\mathcal{M}_{n-3})\to H_j(M_{n-3}))).
		\]
		Using the diagram again, we see that the latter group surjects onto
		\[
		\ker(H_j(X)\to H_j(M_{n-2}))
		\]
		under the inclusion, so
		\[
		H_j(\mathcal{M}_{n-3}\times S^3) \oplus K'\to \ker(H_j(X)\to H_j(M_{n-2}))
		\]
		is surjective. The point is now that (also for $n=6$) by \cref{lem:homologymap}, the image of $H_j(\mathcal{M}_{n-3}\times S^3)\to H_j(X)$ (which is actually $0$ for $n\neq 6$)
		is contained in the image of $K'\to H_j(X)$, because the image of $H_j(\mathcal{M}_{n-3}\times \{pt.\})\to H_j(X)$ is contained in the image of $H_j(F_{n-2}\times S^3)\to H_j(X)$.
		So we have shown that $K'\to \ker(H_j(X)\to H_j(M_{n-2}))$ is indeed surjective. This shows (a).\\

        \item 
		For odd $j\leq n$, the surjectivity of $H_j(\tilde{X})\to H_j(M_{n-2})$ implies that
		\[
		H_j(X_{n-3}')\oplus H_j(\mathcal{M}_{n-3})\to H_j(M_{n-3})
		\]
		is surjective (because for these degrees, the kernel of $h_{j-1}$ injects into $H_{j-1}(\mathcal{M}_{n-3})$, also for $n=6$) and that
		the kernel of $g_{j-1}$ surjects onto the kernel of $h_{j-1}$ (also for $n=6$, as we will show in \cref{rem:cokernel}). This, in turn, implies the surjectivity of $H_j({X})\to H_j(M_{n-2})$
		by the same diagram chase one does to prove the four lemma, and thus (b).\\
		
		\item
		At last for even $4\leq j\leq n-2$, we note that $H_j(Y_{n-3})\oplus K'\to H_j(X_{n-3}')$ is surjective by the top row for $k=n-3$ in \ref{diagram}, so that $H_j(X_{n-3}')\to H_j(X)\to H_j(M_{n-2})$ is the $0$-map.
		For even $4\leq j\leq n-4$, the cokernel of $H_j(X)\to H_j(M_{n-2})$ can thus be identified with $[H_j(M_{n-3})/\text{im}(h_j')]$ because
		$\text{coker}(\ker(g_{j-1})\to \ker(h_{j-1}))=0$, and so can the cokernel of $H_j(\tilde{X})\to H_j(M_{n-2})$. This shows the statement
		for $4\leq j\leq n-4$.\\
		
		For $j=n-2$, we note that $\ker(h_{n-3})$ injects into
		$$K:=\ker(h^*_{n-3}\colon H_{n-3}(\mathcal{M}_{n-3}^*)\to H_{n-3}(M_{n-3}^*)).$$
		Since $H_{n-2}(X_{n-2}^*)\to H_{n-2}(M_{n-2}^*)$ is an isomorphism by \cref{homologyX*}, $K$ is surjected on by
		$$\ker(H_{n-3}(\mathcal{M}^*_{n-3}\times \{pt.\}) \to H_{n-3}(\mathcal{M}^*_{n-3}\times S^2)\to H_{n-3}((X_{n-3}')^*)).$$
        But this can be identified with the kernel of (remember that $Y_{n-3}=Y_{n-3}^*\times T\cong M_{n-3}^*\times S^1 \times T$ and
        $\mathcal{M}_{n-3}=\mathcal{M}_{n-3}^*\times T^{n-2}$)
        \[
        H_{n-3}(\mathcal{M}_{n-3}^*\times \{pt.\})\hookrightarrow H_{n-3}(\mathcal{M}_{n-3})\to H_{n-3}(Y_{n-3})\to H_{n-3}(X_{n-3}'),
        \]
        which implies that $\text{coker}(\ker(g_{n-3})\to \ker(h_{n-3}))=0$.\\
        \end{enumerate}
		For $n=4$, $H_4(X_1)=0$ and $H_2(X_1)\to H_2(M_1)$ is an isomorphism, so $H_2(X_1')\to H_2(M_1)$ is an isomorphism as well,
		and hence $H_2(X)\to H_2(M_2)$ and $H_3(X)\to H_3(M_2)$ are surjective. By the above diagram, the kernel of $H_4(X)\to H_4(M_2)$ can be identified with
		the image of $H_1(F_2)\otimes H_3(S^3)\to H_4(X_1^3)\to H_4(X)$. But this is torsion free for general $n\geq 4$.
		Indeed, since $H_*(Y_{1})$ is generated by $\mathcal{F}_{1}^*\times T^n$
		(which goes to $0$ in $(\mathcal{F}_{1}\times S^{2n-5})\cap X_1^{n-1}$), $H_{\geq 4}(Y_1)\to H_{\geq 4}(X_1^{n-1})$
		and thus $H_{\geq 4}(\mathcal{F}_1^{n-1}) \to H_{\geq 4}(X_1^{n-1})$ is the $0$-map, so $H_{\geq 4}(X_1^{n-1})$ has no torsion by
		\[
		\hdots \to H_*(\mathcal{F}_1^{n-1}) \oplus H_*(Y_1)\to H_*(X_1^{n-1})\to H_{*-1}(\mathcal{F}_1^{n-1}\times S^1)\to \hdots
		\]
		and thus the image of $H_{\geq 1}(F_2)\otimes H_3(S^3)\to H_{\geq 4}(X_1^{n-1})$ has no torsion. This image injects into
		$H_{\geq 4}(X_{n-3}')$ by the diagram in \cref{chain}, so we have shown everything.\\
		
		\item $H_k(X)$ is free abelian for $k\neq n-1$ if and only if it is free abelian for $k\leq n-2$ by Poincaré duality. This holds for $n=4$,
		because the kernel of $H_2(X_1')\to H_2(X)$ is the image of $H_2(Y_1)\to H_2(X_1')$, and the cokernel of both homomorphisms is free abelian
		as can be directly seen from the two diagrams above.\\
		
		Now assume $n\geq 5$. We have that $H_k(X)$ is torsion free if and only if $[H_k(X_{n-3}')\oplus H_k(F_{n-2}\times S^3)]/\text{im}(g_k)$
		is. Now this image $\text{im}(g_k)$ does not depend on $Q$ by \cref{rem:homologyhtilde}, so
		it suffices to show this for $\tilde{X}$ again by the same reasons as in the point before, for which we see this by
		\[
		\hdots \to H_k(\tilde{\mathcal{F}}_{n-2}^{n-1}) \oplus H_k(\tilde{Y}_{n-2})\to H_k(\tilde{X})\to H_{k-1}(\tilde{\mathcal{F}}_{n-2}^{n-1}\times S^1)\to \hdots
		\]
		
		\item Let $n\geq 5$. Since we have the (trivial) bundle $S^1\to Y_{n-3}^*\to M_{n-3}^*$, the corresponding statement holds for $X_{n-3}'$, and
		\[
		H_1((X_{n-3}'\setminus (\partial Z\cap X_{n-3}'))^*)\to H_1((X\setminus \partial Z)^*)
		\]
		is an isomorphism.\\
		If $n=4$, we still have the bundle $S^1\to Y_{1}^*\to M_{1}^*$, although it might not be orientable anymore. However, the fiber becomes nullhomotopic
		in $(X_1\setminus Z)^*$ the moment this set contains a connected component of $(\mathcal{F}_1^3)^*$, which it does not do if and only if any three edges
		emerging from a vertex belong to an ineffective subgraph. If the latter is true, then the bundle $S^1\to Y_{1}^*\to M_{1}^*$ becomes orientable, since
		it is the restriction of the orientable, not necessarily trivial, bundle $S^1\to Y_{2}^*\to M_{2}^*$, which then exists by \cref{fol}.\\
		
		\item Every connected component of $\partial Z$ is equivariantly diffeomorphic
		to a product of $T^{n-1}$ and a homology sphere with some ineffectice action of $T^{n-1}$ on itself (see \cref{homologyFkn-1}).
		Now use the statement from before, the long exact sequence
		\[
		\hdots \to H_{*+1}(X^*)\to H_{*}(\partial Z^* \times S^1)\to H_{*}(\partial Z^*)\oplus H_{*}((X\setminus \partial Z)^*)\to H_{*}(X^*)\to \hdots
		\]
		and that the interesting homology of $X^*$ is concentrated in degrees $2$ and $n-2$ (it also follows that the interesting homology of $(X\setminus \partial Z)^*$
		is concentrated in degree $2$, $n-2$ and $n-1$, as long as not every GKM subgraph of covalence one is ineffective).
	\end{enumerate}
\end{proof}
\begin{remark}\label{rem:cokernel}
		We show that, under our additional assumption in \cref{propX}, it holds for $n=6$ that the cokernel of $\ker(g_{j})\to \ker(h_j)$ in
		\begin{equation*}
			\begin{tikzcd}
				\hdots \overset{g_{j}}{\to} H_j(X_{n-3}')\oplus H_j(F_{n-2}\times S^3) \arrow{r} \arrow{d} & H_j(Q) \arrow{r} \arrow{d} & H_{j-1}(\mathcal{M}_{n-3}\times S^3) \arrow{d} \overset{g_{j-1}}{\to} \hdots\\
				\hdots \overset{h_{j}}{\to} H_j(M_{n-3})\oplus H_j(F_{n-2}) \arrow{r} & H_j(M_{n-2}) \arrow{r} & H_{j-1}(\mathcal{M}_{n-3}) \overset{h_{j-1}}{\to} \hdots
			\end{tikzcd}
		\end{equation*}
does not depend on $Q$.
Indeed, $X$ and $\tilde{X}$ have the same rational homology in degrees $\leq 5$,
because their rational Serre spectral sequences are the same ($H_2((X_3')^*)\to H_2(X^*)$ and $H_2((X_3')^*)\to H_2(\tilde{X}^*)$ are isomorphisms,
so the differentials $d^2_{2,*}$ of their sequences are the same by naturality).\\
Now we have a splitting of $H_j(\mathcal{M}_{3}\times S^3)$ (remember that $j\leq 5$) into $K_1\oplus K_2\oplus K_3$, where
\[
K_1=H_3(\mathcal{M}_{3}^*)\otimes H_{j-3}(T^4), \; K_2=H_0(\mathcal{M}_3)\otimes H_{j-3}(T^4)\otimes H_3(S^3), \; K_3=H_0(\mathcal{M}_3)\otimes H_j(T^4).
\]
Since $K_1\to H_j(X_3')$ factors through $H_j(Y_{3})\to H_j(X_{3}')$ and $K_2\to H_j(X_3')$
factors through $H_j(X_1')\to H_j(X_{3}')$, the images of these two maps are disjoint,
and so the kernel of $g^1_j$ in $g_j=g^1_j\oplus g^2_j$ is given by
\[
\ker(K_1\to H_j(X_3'))\oplus \ker(K_2\to H_j(X_3'))\oplus K_3 =:K_1'\oplus K_2'\oplus K_3.
\]
However, $K_3$ injects into $H_j(F_4\times S^3)$ and the image in there is disjoint with the image of
$K_1\oplus K_2$, so that the kernel of $g_j$ is the kernel of $g^2_j$ restricted to $K_1'\oplus K_2'$.\\

Now, the image of $K_1'\oplus K_2'\to H_*(F_4\times S^3)$ in the case of a torus manifold (that is, the map $\tilde{h}_j$ in
\cref{rem:homologyhtilde} is $0$) is precisely $K_2'$, since $K_1'$ is sent to $0$. In particular, the rational dimension of the image in every degree is $\dim_{\Q} K_2'$
(in that degree) and the rational dimension of the kernel in every degree is $\dim_{\Q} K_1'$ (in that degree).\\

For general $\tilde{h}_j$, the rational dimension of the image $K_1'\oplus K_2'\to H_j(F_4\times S^3)$ is at least $\dim_{\Q} K_2'$, because $K_2'$
injects into there, and so the rational dimension of the kernel is at most $\dim_{\Q} K_1'$ (again, in every degree). If $X$
and $\tilde{X}$ have the same rational homology, the image of $K_1'\oplus K_2'\to H_j(F_4\times S^3)$ (for the $\tilde{h}$ belonging to $X$) has the same dimension as for $\tilde{h}=0$,
which implies that the image of
$$\tilde{h}\otimes \text{id}\colon K_1'\to [H_0(\mathcal{M}_3) \otimes H_3(S^3)]\otimes H_{*-3}(T^4)$$
is contained in $K_2'$, with rational coefficients.\\
As we saw in the discussion for $n=4$ in the last part of point (2) of the proof of \cref{propX}, the image of
\[
H_0(\mathcal{M}_{n-3})\otimes H_3(S^3)\otimes H_{j-3}(T^4)\to H_j(X_3')
\]
with integer coefficients is torsion free in degrees
$\geq 3$, so the image of $\tilde{h}\otimes \text{id}$ is contained in $K_2'$ when taking integer coefficients, too.\\
All in all, this shows that the cokernel of $\ker(g_{j})\to \ker(h_j)$ does not depend on $Q$ for $j\leq 5$. 
\end{remark}

\newpage

\section{A sufficient criterion for equivariant formality}\label{suffEF}
Here, we want to formulate and prove a sufficient criterion for equivariant formality of spaces in general position. Let us, at first, explicitly state and prove the theorem for
$n=3$, which turns the assumption on general position into the ordinary GKM condition. The arguments for general $n$ follow the same line, but become way more technical
due to the fact that there is less control over the isotropy submanifolds.
\begin{theorem}\label{theosuffEF3}
	Consider a 2-independent action of $T^2$ on the compact manifold $M$ of dimension $6$, such that
	\begin{itemize}
		\item the orbit space is a homology sphere over $\Z$.
		\item for every finite group $H$, every connected component of $M^H$ contains a fixed point.
		\item the orbit space of an arbitrary isotropy submanifold is a disk.
	\end{itemize}
	Then $M$ is equivariantly formal over $\Z$.
\end{theorem}
In the following (also for general $n$), we write $Z$ for the isotropy submanifolds of $N:=M\setminus (M_{n-2}'\setminus X)$. Every component hits $X$ non-trivially
by the assumption that such a component contains a fixed point, so we can set $\partial Z=X\cap Z$, which is a union of $S^1\times T^2$'s for $n=3$.\\
Let $C$ be a connected component of $Z$, and $H\subset T$ be its stabilizer. The boundary $E$ of a small closed neighborhood of $C$ is equivariantly diffeomorphic to $C\times_H S^1$,
since $C^*=D^2$ is contractible and the normal fiber over a torus orbit is of that form. This directly implies that $H$ is cyclic, because it necessarily acts freely on $S^1$.
Note also that the $T$-action on $E$ is free and that the orbit space is homotopy equivalent to $S^1$, so we may also write $E=C^*\times S^1\times T$,
$T$ acting only on the right factor. In this description, the natural map $E\to C$ is a bundle map (viewing both spaces as nonprincipal $T^2$-bundles), but restricted to $T$ it is only a covering,
not a diffeomorphism! In particular, the map in homology between the fibers (see \cref{lem:naturalityspectralsequence}) is not an isomorphism, but only an injection.\\
The corresponding map on orbit spaces $E^*=C^*\times S^1\to C^*$ in this description, however, is the usual projection. We will also write $E=Z^*\times S^1\times T$
for the boundary of a small closed neighborhood of $Z$ in $N$.\\
Before proving \cref{theosuffEF3}, we need a small lemma, first. This is formulated and proven for general $n$, under the assumption that every connected component of an isotropy submanifold contains a fixed point; the normal bundle of $Z^*$ in $N^*$ is then still a disk bundle over $Z^*$
(though not necessarily trivial).
\begin{lemma}\label{lem:fundamentalgroup}
	$\pi_1(M)\to \pi_1(M^*)$ is an isomorphism.
\end{lemma}
\begin{proof}
	We consider $M_{n-2}'$ and $X:=X_{n-2}$ for the manifold $M$, see \cref{equneigh}.
	We set $N:=M\setminus (M_{n-2}'\setminus X)$. It suffices to show that $\pi_1(N)\to \pi_1(N^*)$ is an isomorphism by dimensional reasons.
	It is standard theory that this map is surjective, and we may assume by transversality that
	a loop $\gamma$ in its kernel only hits free orbits $N_{free}$. If $\gamma$ is in the kernel of $\pi_1(N_{free})\to \pi_1(N_{free}^*)$, it is in a torus orbit and we are done
	(because these can be homotoped to a point in $M$). If not, then the image $\gamma'$ of $\gamma$ is in the kernel of $\pi_1(N_{free}^*)\to \pi_1(N^*)$. This is generated by the $S^1$-fibers
	of the boundary of the normal bundle of $Z^*$ in $N$, so $\gamma'$ can thus be homotoped into $X_{free}^*$ (through $N_{free}^*$),
	to be in the kernel of $\pi_1(X_{free}^*)\to \pi_1(X^*)$. The kernel of $\pi_1(X_{free}^*)\to \pi_1(X^*)$ is
	generated by the fiber of $S^1\to Y_{n-2}^*\to B_{n-2}$ if $n=3$ (see \cref{singfol}) or can be homotoped along $Z^*$ into the fiber of
	$S^1\to Y_{n-3}^*\to B_{n-3}$ if $n\geq 4$,
	and so $\gamma$ is a loop in the fiber of $T^n\to Y_{n-2}\to B_{n-2}$ (or $T^n\to Y_{n-3}^*\to B_{n-3}$).
	This fiber can be homotoped (through $M$) to a fixed point, so we are done. 
\end{proof}
Now we can come to the proof of \cref{theosuffEF3}.
\begin{proof}
	By Poincaré duality, $H^{odd}(M)=H_{odd}(M)$, and we show that the latter is $0$.
	We have already shown that $\pi_1(M)=\pi_1(M^*)$, so we only need to show that $H_3(M)=0$. Set $N:=M\setminus (M_1'\setminus X)$,
	where $M_1'$ and $X:=X_1$ are as in \cref{equneigh}.
	By $H_*(M^*)=H_*(S^{4})$ and the Mayer Vietoris sequence belonging to $M^*=M_{1}'^*\cup N^*$,
	we see that the interesting homology of $N^*$ is concentrated in degree $2$ ($H_3(N)=H^1(N^*/X^*)=0$ because of Lefschetz duality) and that $H_2(X^*)\to H_2(N^*)$ is an isomorphism.\\
	In view of $H_3(M_1)=0$ and $H_2(X)\to H_2(M_1)$ being an isomorphism (see \cref{X_1^*}), need to show that $H_3(X)\to H_3(N)$ is surjective because of
	\[
	\hdots \to H_3(X)\to H_3(M_1)\oplus H_3(N)\to H_3(M)\to H_2(X)\to H_2(M_1)\oplus H_2(N)\to \hdots
	\]
	If the action on $N$ was free, this would be an immediate consequence of
	the isomorphism (we only need it to be a surjection) $H_2(X^*)\to H_2(N^*)$ and the Serre spectral sequence. Indeed, $E^2_{p,*}(X)\to E^2_{p,*}(N)$ is a surjection
	for $p=2$ and an injection (even an isomorphism) for $p=0$, so $\ker(d^2_{2,*}(X))\to \ker(d^2_{2,*}(N))$ is a surjection
	(remember that, for homology, $d^r_{p,q}$ goes to $E^r_{p-r,q+r-1}$) and thus $E^3_{2,*}(X)\to E^3_{2,*}(N)$ is. Also,
	$E^3_{0,*}(X)\to E^3_{0,*}(N)$ is an injection, then. Due to degree reasons, $E^3_{2,*}(X)=E^{\infty}_{2,*}(X)$
	and the same for $N$. Now the claim follows by $H_3(N)=E_{2,1}^{\infty}(N)$.\\
	
	In case of the action not being free, we denote by $Z$ the isotropy submanifolds
	in $N$ (whose orbit spaces are disks) and consider the diagram for the orbit spaces
	\[
	\begin{tikzcd}
		H_2((X\setminus \partial Z)^*) \oplus H_2(\partial Z^*)\arrow{r} \arrow{d} & H_2(X^*) \arrow{r}{\partial_2} \arrow{d} &
		H_{1}(\partial Z^*\times S^1) \arrow["i_1",r] \arrow{d} & H_{1}((X\setminus \partial Z)^*) \oplus H_{1}(\partial Z^*) \arrow{d}\\
		H_2((N\setminus Z)^*) \oplus 0\arrow[hookrightarrow]{r} & H_2(N^*) \arrow{r}{\partial_2} & H_{1}(Z^*\times S^1) \arrow["i_2",r] & H_{1}((N\setminus Z)^*)\oplus 0
	\end{tikzcd}
	\]
	and the diagram for the total spaces
	\[
	\begin{tikzcd}
		H_3(X\setminus \partial Z) \oplus H_3(\partial Z)\arrow{r} \arrow{d} & H_3(X) \arrow{r}{\partial_3} \arrow{d} &
		H_{2}(\partial E) \arrow{r} \arrow{d} & H_{2}(X\setminus \partial Z) \oplus H_{2}(\partial Z) \arrow{d}\\
		H_3(N\setminus Z)\oplus 0 \arrow[hookrightarrow]{r} & H_3(N) \arrow{r}{\partial_3} & H_{2}(E) \arrow{r} & H_{2}(N\setminus Z)\oplus H_{2}(Z)
	\end{tikzcd}
	\]
	We wrote $0$ for the terms of $H_*(Z^*)$ respectively $H_*(Z)$ that are $0$, because $Z^*$ is a union of discs and $Z$ is topologically a union of $D^2\times T^2$'s.\\	
	In the above diagram, we see by a simple diagram chase and usage of $H_2(X^*)\to H_2(N^*)$ being an isomorphism
	\begin{itemize}
		\item the kernel of $i_2$ is isomorphic to the kernel of $i_1$ via the vertical map.
		\item $H_2((X\setminus \partial Z)^*)\to H_2((N\setminus Z)^*)$ is surjective, because of the last statement, and because $H_2((N\setminus Z)^*)\to H_2(N^*)$ is injective.
	\end{itemize}
	Thus, $H_3(X\setminus \partial Z)\to H_3(N\setminus Z)$ is surjective by the same type of argument as when the action is free on $N$
	(we only needed $H_2((X\setminus \partial Z)^*)\to H_2((N\setminus Z)^*)$ to be surjective). In order to conclude the theorem, we have to show that $K_1\to K_2$, where
	\[
	K_1:=\ker(H_2(\partial E)\to H_2(X\setminus \partial Z)\oplus H_2(\partial Z)),\; K_2:= \ker(H_2(E)\to H_2(N\setminus Z)\oplus H_2(Z)),
	\]
	is surjective (this suffices by the same proof of the four-lemma about surjectivity, although the assumptions are slightly different).
	Since the occuring maps on spaces are bundle maps, we may use the Serre spectral sequence to understand this. Moreover,
	under $H_2(\partial E)\to H_2(E)$, we can and we will view $K_2$ as contained in
	$$H_0(\partial Z)\otimes H_2(S^1\times T) \subset H_2((\partial E)_1)$$
	(see \cref{rem:groupext} for the notation) from now on.\\
	We want to understand the kernel $K_1'$ of
	\[
	E^{\infty}_{1,1}(\partial E)\to E^{\infty}_{1,1}(X\setminus \partial Z)\oplus E^{\infty}_{1,1}(\partial Z),
	\]
	first. By degree reasons, $E^{\infty}_{1,1}=E^2_{1,1}$ for all occuring terms. It follows that $K_1'$
	is contained in
	$$H_0(\partial Z)\otimes H_1(S^1) \otimes H_1(T),$$
	because $\ker(H_1(\partial E^*)\to H_1(\partial Z^*))$ is contained in $H_0(\partial Z)\otimes H_1(S^1)$.\\
	Also, $K_1'$ is mapped isomorphically (because $\ker(i_1)\to\ker(i_2)$ is an isomorphism) to the kernel $K_2'$ of
	$$E_{1,1}^{\infty}(E)\to E_{1,1}^{\infty}(N\setminus Z)\oplus E_{1,1}^{\infty}(Z),$$
	which is contained in
	$$H_0(Z)\otimes H_1(S^1) \otimes H_1(T),$$
	so that we may view both $K_1'$ and $K_2'$ to be the same subgroup in
	$$H_0(\partial Z) \otimes H_1(S^1) \otimes H_1(T)\subset E^{\infty}_{1,1}(\partial E).$$
	Of course, these groups are not necessarily equal to $K_1$ and $K_2$, but, using \cref{rem:groupext}, they are related to them via the following diagram
	\[
	\begin{tikzcd}
		0 \arrow{r} & E^{\infty}_{0,2}(\partial E) \arrow{r} \arrow{d} & H_2((\partial E)_{(1)}) \arrow["j_1", r] \arrow{d} & E^{\infty}_{1,1}(\partial E) \arrow{r} \arrow["j_2",d] & 0\\
		0 \arrow{r} & E^{\infty}_{0,2}(\partial Z)\oplus E^{\infty}_{0,2}(X\setminus \partial Z) \arrow{r} \arrow{d} & G_1 \arrow{r} \arrow{d} &
		E^{\infty}_{1,1}(\partial Z)\oplus E^{\infty}_{1,1}(X\setminus \partial Z) \arrow{r} \arrow{d} & 0\\
		0 \arrow{r} & E^{\infty}_{0,2}(Z)\oplus E^{\infty}_{0,2}(N\setminus Z) \arrow{r} & G_2 \arrow{r} & E^{\infty}_{1,1}(Z)\oplus E^{\infty}_{1,1}(N\setminus Z) \arrow{r} & 0
	\end{tikzcd}
	\]
	We used the notation
	$$G_1=H_2((\partial Z)_1)\oplus H_2((X\setminus \partial Z)_1), \quad G_2=H_2(Z_1)\oplus H_2((N\setminus Z)_1).$$
	By definition of these, we have injections
	$$G_1\hookrightarrow H_2(\partial Z)\oplus H_2(X\setminus \partial Z), \quad G_2\hookrightarrow H_2(Z)\oplus H_2(N\setminus Z).$$
	Therefore, it suffices to show that
	any $x$ in $K_2\subset H_2((\partial E)_{(1)})$ (which is $0$ in $G_2$, then) is $0$ in $G_1$.\\
	We know that $j_1(x)\in K_2=K_1$, so $j_2(j_1(y))=0$. It follows that the image of $x$ in $G_1$ comes from
	$E^{\infty}_{0,2}(\partial Z)\oplus E^{\infty}_{0,2}(X\setminus \partial Z)$,
	but this injects into $E^{\infty}_{0,2}(Z)\oplus E^{\infty}_{0,2}(N\setminus Z)$ (as argued in the case $Z=\emptyset$). So $y$ is $0$ in $G_1$, as well, and we have shown the claim
	and thus the whole assertion.
\end{proof}
Let us now turn to the general case. The main difference is that we do not assume that $Z^*$ is acyclic over $\Z$, so there are some technical difficulties
to overcome. For example, we do not know if the normal bundle of $Z^*$ in $N^*$ is trivial.
\begin{theorem}\label{theosuffEF}
	Consider an action in general position of $T=T^{n-1}$ on the compact manifold $M$ of dimension $2n\geq 8$ such that
	\begin{itemize}
		\item for all closed subgroups $H\subset T$, every connected component of $M^H$ intersects $M^T$ non-trivially.
		\item the orbit space $M/T$ has the homology of a sphere (and is thus a homology sphere).
		\item $M_{n-2}^*$ is $n-3$-acyclic over $\Z$.
		\item every face-submanifold of codimension $4$ is equivariantly formal.
		\item for any isotropy submanifold $Q$ fixed by $\Z_p\subset T$, corresponding to an ineffective subgraph of covalence $1$,
		the map $H_*(Q^*)\to H_*(Q^*)$ given by multiplication with $p$ is an isomorphism in positive degrees.
	\end{itemize}
    Then the action is equivariantly formal over $\Z$.
\end{theorem}
\begin{remark}
	If every $Q$ fulfills the conditions of the theorem (which is equivalent to Tor$(H_{\geq 1}(Q^*);\Z_p)=0$, of course),
	then we call $Q$ \textit{admissible}. This particularly implies that $Q^*$ is rationally acyclic.
	We will argue later why the other conditions on $Q$ are absolutely necessary.\\
	If every $Q^*$ is acyclic over $\Z$, the arguments become less technical and shorten dramatically. For example, normal bundles as well
	as the principal $T^{n-1}$-bundles over the $Q^*$'s are trivial, then, and one could argue very similarly to the case $n=3$.
\end{remark}
The proof of this will be done in this chapter. By the conditions we imposed and the use of \cref{torusManifold},
we may use the results in \cref{local} throughout. For example, by \cref{homologyX*}, the reduced homology of $X^*$ is concentrated in degrees $2$ and $n-2$ (and $n$).
Further, since $(M_{n-2}^*) \cup (N^*)$ is a sphere with $(M_{n-2}^*) \cap (N^*)=X^*$, Mayer Vietoris implies that
\[
H_j(X^*)\to H_j(M_{n-2}^*)\oplus H_j(N^*)
\]
is an isomorphism for $j=2$ and $j=n-2$ and that therefore the reduced homology of $N^*$ is concentrated in degree $2$
(that $H_{n-1}(N^*)=0$ follows from Lefschetz duality and $H^1(N,X)=0$).\\

We want to start with some basic properties, where we only use the fact that $M^*$ is a homology sphere.
\begin{lemma}\label{basprop}
	Let $Z$ be the union of all isotropy submanifolds in $N$, and let $E$ be the boundary of a closed neighborhood of them.
\begin{enumerate}
	\item $H_2((X_1^{n-1})^*)\to H_2(N^*)$ is an isomorphism for $n\geq 4$ and $H_2((X_1^{n-1}\setminus \partial Z)^*)\to H_2((X\setminus \partial Z)^*)$ is surjective for $n\geq 5$.\\
	
	\item For $j\geq 2$, the map $H_j(E^*)\to H_j(Z^*)\oplus H_j((N\setminus Z)^*)$ is an isomorphism in torsion parts, and we have $H_j(E^*)\cong H_j(Z^*\times S^1)$.\\
	
	\item Set $I$ to be the homology of the 2-skeleton of a CW decomposition of $(X_1^{n-1}\setminus \partial Z)^*$. The map
	\[
	H_{*}(E^*)\oplus I\to H_{*}(Z^*)\oplus H_{*}((N\setminus Z)^*)
	\]
	induced by $E^*\to Z^*$, $E^*\to (N\setminus Z)^*$ and $(X_1^{n-1}\setminus \partial Z)^*\to (N\setminus Z)^*$ is surjective.
	Its kernel $K$ is free abelian and concentrated in degree $1$ and $2$. In fact, its intersection with $H_*(E^*)$ is isomorphic to
	\[
	\ker[H_1(\partial Z^*\times S^1)\to H_1(\partial Z^*) \oplus H_1((X_1^{n-1}\setminus \partial Z)^*)]
	\]
	under the inclusion in $E^*$.
\end{enumerate}	
\end{lemma}
\begin{proof}
	For the first point we consider
	\[
	\begin{tikzcd}
		\hdots \arrow{r} & H_*((\mathcal{F}_1^{n-1})^*) \oplus H_*(Y_1^*)\arrow{r} \arrow{d} & H_*((X_1^{n-1})^*) \arrow{r}{\partial_*} \arrow{d} &
		H_{*-1}((\mathcal{F}_1^{n-1})^*\times S^1) \arrow{r}{g_{*-1}} \arrow{d} & \hdots\\
		\hdots \arrow{r} & H_*((\mathcal{F}_{n-3}^{n-1})^*) \oplus H_*(Y_{n-3}^*)\arrow{r} & H_*((X_{n-3}^{n-1})^*) \arrow{r}{\partial_*} &
		H_{*-1}((\mathcal{F}_{n-3}^{n-1})^*\times S^1) \arrow{r}{h_{*-1}} & \hdots\\
	\end{tikzcd}
	\]
    For the last two points, we look at (writing $\partial Z$ for $\partial Z\cap X_1^{n-1}$)
	\[
	\begin{tikzcd}
		\hdots \arrow{r} & H_j((X_1^{n-1}\setminus \partial Z)^*) \oplus H_j(\partial Z^*)\arrow{r} \arrow{d} & H_j((X_1^{n-1})^*) \arrow{r}{\partial_j} \arrow{d} &
		H_{j-1}(\partial Z^*\times S^1) \arrow{r} \arrow{d} & \hdots\\
		\hdots \arrow{r} & H_j((N\setminus Z)^*)\oplus H_j(Z^*) \arrow{r} & H_j(N^*) \arrow{r}{\partial_j} & H_{j-1}(E^*) \arrow{r} & \hdots
	\end{tikzcd}
	\]
\begin{enumerate}
	\item 
	Let $n\geq 5$.
	The upper $\partial_*$ is an injection in degree $2$, and its image is in $H_0((\mathcal{F}_1^{n-1})^*)\otimes H_1(S^1)$, which injects
	into $H_1((\mathcal{F}_{n-3}^{n-1})^*\times S^1)$. It follows that the image of the middle vertical map is disjoint with the image of the bottom left map,
	which is at the same time the kernel of $H_2((X_{n-3}')^*)\to H_2(X^*)$. This shows the injectivity of
	$f\colon H_2((X_1^{n-1})^*)\to H_2(X^*)\overset{\cong}{\to} H_2(N^*)$.\\
	
	To show the surjectivity of $f$, we note that the image of
	$$H_0((\mathcal{F}_1^{n-1})^*)\otimes H_1(S^1)\to H_*((\mathcal{F}_1^{n-1})^*) \oplus H_*(Y_1^*)$$
	is generated by the $S^1$-fiber of $Y_1^* \cong B_1\times S^1$, and that this image gets mapped isomorphically to the homology generated by the $S^1$-fiber
	of $Y_{n-3}^*\cong B_{n-3}\times S^1$. It follows that $\ker(g_1)\to \ker(h_1)$ is surjective. Thus, the cokernel of the middle vertical map
	is generated by the image of the bottom left map, which is trivial for $n\neq 5$. For $n=5$, the image of the bottom left map vanishes in $H_2(X)$.
	In both cases, it follows that $f$ is surjective.\\

	For $n=4$, we note that $H_2((X_1^{3})^*)\to H_2(X^*)$ is a surjection onto the kernel of $H_2(X^*)\to H_2(M_2^*)$ by the diagram in \cref{chain}.
	Since $H_2(X_2^*)\to H_2(N^*)\oplus H_2(M_2^*)$ is an isomorphism,
	the kernel of $H_2(X_2^*)\to H_2(M_2^*)$ gets mapped isomorphically to $H_2(N^*)$. This shows the claim here.\\
	
	\item The first statement comes from the bottom row of the diagram. For the second statement, we look
	at the spectral sequence associated to $S^1\to E^*\to Z^*$ and claim that this collapses at the second page.
	It certainly does so at the third page by degree reasons, so we only need to show that $K_p:=\ker(d^2_{p,0})=H_p(Z^*)$.
	We have a bundle map $E^*\to Z^*$ which induces homomorphisms between the second and third page. In particular, $E^2_{p,0}\to H_p(Z^*)$ is an isomorphism,
	$E^{\infty}_{p,0}=E^3_{p,0}=K_p\to H_p(Z^*)$ is the inclusion and the homomorphism restricted to $E^2_{p,1}$ and $E^3_{p,1}$ is $0$. Now $H_p(E^*)\to H_p(Z^*)$ fits into
	\[
\begin{tikzcd}
	0 \arrow{r} & E^3_{p-1,1} \arrow{r} \arrow{d} & H_p(E^*) \arrow{r} \arrow{d} & K_p \arrow{r} \arrow{d} & 0\\
	0 \arrow{r} & 0 \arrow{r} & H_p(Z^*) \arrow{r} & H_p(Z^*) \arrow{r} & 0
\end{tikzcd}
\]
and so $K_p\to H_p(Z^*)$ needs to be surjective since $H_p(E^*)\to H_p(Z^*)$ is. That the above sequence splits is now
a consequence of the isomorphism
$$H_j(E^*)\to H_j(Z^*)\oplus H_j((N\setminus Z)^*)$$ in torsion parts.\\

	\item The first statement is evident by the diagram and the surjectivity of $H_2((X_1^{n-1})^*)\to H_2(N^*)$ and $I\to H_{\leq 2}((X_1^{n-1}\setminus \partial Z)^*)$.\\
	For the second one, we note that the kernel of $H_1(E^*)\to H_1(Z^*)$ is generated by the fibers (one for every connected component of $Z^*$) of $E^*\to Z^*$. Now the statement follows from
	$H_2((X_1^{n-1})^*)\to H_2(N^*)$ being surjective.
\end{enumerate}
\end{proof}
We want to show that the odd integer homology of $M$ vanishes by considering the Mayer Vietoris sequence of the triple
$(X, M_{n-2}, N)$, where $N=M\setminus (M_{n-2}\setminus X)$. Here, the odd homology of $M$ vanishes if and only if
\begin{itemize}
	\item $H_{even}(X)\to H_{even}(M_{n-2})\oplus H_{even}(N)$ is injective.
	\item $H_{odd}(X)\to H_{odd}(M_{n-2})\oplus H_{odd}(N)$ is surjective.
\end{itemize}
For $n=4$, we only need to show the surjectivity for degree $3$, since already $H_2(X)\to H_2(M_2)$ is injective by \cref{propX}.\\
For $n\geq 5$, we will not show these statements in general. Instead, we will show them for degrees $\leq n$ and then show
that $H_{\leq n-2}(M)$ is torsion free in even degrees, so that the odd homology of degree $\geq n+1$ is torsion free (and thus $0$) as well.\\

The injectivity is clear immediately for all even degrees $j\leq n$, since we saw in \cref{propX} that
the kernel of $H_j(X)\to H_j(M_{n-2})$ is free abelian and the image of $H_*(X_1^{n-1})\to H_*(X)$, which
injects into $H_*(N)$ when taking $\Q$-coefficients. Indeed, $H_2((X_1^{n-1})^*)\to H_2(N^*)$ is an isomorphism by \cref{basprop}, and now we can use the rational Serre spectral sequence, see \cref{serreQ}, as well as the fact that the sequence for $N$ collapses
at the third page due to degree reasons.\\

That $H_{j}(M)$ is torsion free for even $j\leq n-2$ follows from the fact that
$i\colon H_{j}(X)\to H_{j}(M_{n-2})\oplus H_{j}(N)$ is injective by the previous discussion,
and that $H_*(X_1^{n-1})\to H_*(N)$ is surjective for all degrees $\leq n$ (which we show soon). If the latter holds,
then the cokernel of $i$ is the cokernel of $H_{j}(X)\to H_{j}(M_{n-2})$ (which is torsion free by \cref{propX}) for $j\geq 4$,
because $H_{\geq 3}(X_1^{n-1})\to H_{\geq 3}(M_1)=0$ is the $0$-map.\\
For degree $2$, we then have that both $H_2(X)\to H_2(N)$ and $H_2(X)\to H_2(M_{n-2})$ (see \cref{propX}) are isomorphisms, so $H_2(M)=H_2(X)$, which is torsion free by \cref{propX}.\\

Now let us come to the main task, which is showing the surjectivity. We already know that $H_j(X)\to H_j(M_{n-2})$ is a surjection for odd $j$ by \cref{propX}, so now we want to show
this for $H_j(X)\to H_j(M_{n-2})\oplus H_j(N)$, for which it suffices to show that $H_j(X_1^{n-1})\to H_j(N)$ is surjective for all $2\leq j\leq n$
(we have $\pi_1(M)=\pi_1(M^*)$ by \cref{lem:fundamentalgroup}, so we already now $H_1(M)=0$). We saw in \cref{basprop} that
\[
H_{\geq 1}(E^*)\oplus I_{\geq 1}\to H_{\geq 1}(Z^*)\oplus H_{\geq 1}((N\setminus Z)^*)
\]
is surjective and that its kernel $K$ is free abelian and concentrated in degree $1$ and $2$. We set $\tilde{I}=I\otimes H_*(T)$ and want to show that
\[
h^2\colon E^2_{p,*}(E)\oplus \tilde{I}_{p,*}\ \to E^2_{p,*}(Z)\oplus E^2_{p,*}(N\setminus Z)
\]
(note that the maps on these spaces are bundle maps) is also surjective for $p\geq 1$, its kernel $\tilde{K}$ is $K\otimes H_*(T)$ and that it is injective for $p=0$.\\
Similar to the discussion at the beginning of \cref{suffEF} for $n=3$,
the map in homology between the $T^{n-1}$-fibers of $E$ and $Z$ is only an injection, not an isomorphism,
because the fibers of $E$ are not mapped diffeomorphically onto the fibers of $Z$. Indeed,
for a connected component $E'$ of $E$ and the corresponding component $Z'$ of $Z$ fixed by $\Z_p\subset T$, the fiber map
$T^{n-1}\to T^{n-1}$ belonging to $E'\to Z'$
is given by $T^{n-1}\to T^{n-1}/ \Z_p=T^{n-1}$. We may assume that $\Z_p\subset S^1\times \{e\}\subset S^1\times T^{n-2}$,
so the corresponding map in homology is given by
\[
H_*(S^1)\otimes H_*(T^{n-2})\to H_*(S^1)\otimes H_*(T^{n-2}), \quad x\otimes y \mapsto p x \otimes y=p(x\otimes y).
\]
This implies that $E^2_{\geq 1,*}(E')\to E^2_{\geq 1,*}(Z')$ is of the form
$$(H_*(E'^*)\to H_*(Z'^*)\overset{\cdot p}{\to} H_*(Z'^*))\otimes \text{id}_{H_*(T^{n-1})}.$$
Since multiplication with $p$ in $H_{\geq 1}(Z')$ is an isomorphism by assumption, we deduce that $h^2$ is 
a surjection and its kernel $\tilde{K}$ is $K\otimes H_*(T)$. Moreover, $d^2$ vanishes on $\tilde{K}$ by the compatibility of the differentials of
both sides and the injectivitiy for the $E^2_{0,*}$-groups,
and the image of $d^2_{\geq 2,*}$ of the left side intersects $\tilde{K}$ trivially because $\tilde{K}$ is free abelian. It follows that
\[
(E^2(E)\oplus \tilde{I})/\tilde{K} \to E^2(Z)\oplus E^2(N\setminus Z)
\]
is an injective chain homomorphism, and its cokernel is generated by $H_0(\partial Z)\otimes H_*(T)\to E^2_{0,*}(Z)$
(so that it is even an isomorphism in the $E^2_{\geq 1,*}$-entries).\\
Thus, passing to their homologies, we have that
\[
h^3\colon (E^3(E)\oplus \tilde{I}^3)/\tilde{K} \to E^3(Z)\oplus E^3(N\setminus Z),
\]
where $\tilde{I}^3:=\ker(d^2)_{\tilde{I}}$, is again an injective chain homomorphism, and its cokernel is generated by $H_0(\partial Z)\otimes H_*(T)\to E^3_{0,*}(Z)$.\\
We can continue arguing like this and then see that ($\tilde{I}^3=\tilde{I}^{\infty}$ by degree reasons)
\[
h^{\infty}\colon (E^{\infty}(E)\oplus \tilde{I}^{\infty})/\tilde{K} \to E^{\infty}(Z)\oplus E^{\infty}(N\setminus Z)
\]
fulfills all the same statements.\\
In particular, we have that
\[
H_*(E)\oplus H_*(X_1^{n-1}\setminus \partial Z) \oplus H_*(\partial Z)\to H_*(Z)\oplus H_*(N\setminus Z)
\]
is a surjection by basic properties of group extensions. We also get the following lemma.
\begin{lemma}\label{kerdelinfty}
	It holds $\tilde{K}\cap E^{\infty}(E)=(K\cap H_*(E^*))\otimes H_*(T) \subset E_{1,*}^{\infty}(E)$.\\
	In particular,
	\[
	\ker[E_{1,*}^{\infty}(\partial Z\times S^1)\to E_{1,*}^{\infty}(\partial Z)\oplus E_{1,*}^{\infty}(X_1^{n-1}\setminus \partial Z)]\to \tilde{K}\cap E^{\infty}(E)
	\]
	is an isomorphism.
\end{lemma}
\begin{proof}
	The first statement is just a result of the discussion above. The second follows from the first and the fact that
	\[
	\ker[H_1(\partial Z^*\times S^1)\to H_1(\partial Z^*) \oplus H_1((X_1^{n-1}\setminus \partial Z)^*)]\to K\cap H_*(E^*)
	\]
	is an isomorphism by \cref{basprop}.
\end{proof}
To understand
\[
K':=\ker(H_*(E) \to H_*(Z)\oplus H_*(N\setminus Z))
\]
we prove the following lemma.
\begin{lemma}\label{kerdel}
	Writing $\partial Z$ for $\partial Z\cap X_1^{n-1}$, $K'$ is in the image (under the inclusion) of
	\[
	\ker[H_*(\partial Z\times S^1)\to H_*(\partial Z)\oplus H_*(X_1^{n-1}\setminus \partial Z)].
	\]
\end{lemma}
\begin{proof}
First, we look at
\[
\begin{tikzcd}
	0 \arrow{r} & E^{\infty}_{0,*}(\partial Z\times S^1) \arrow{r} \arrow["\cong", d] & H_1 \arrow{r} \arrow{d} & E^{\infty}_{1,*-1}(\partial Z\times S^1) \arrow{r} \arrow{d} & 0\\
	0 \arrow{r} & E^{\infty}_{0,*}(E)\arrow{r} \arrow{d} & G_1 \arrow["f",r] \arrow{d} & E^{\infty}_{1,*-1}(E)\arrow{r} \arrow["g",d] & 0\\
	0 \arrow{r} & E^{\infty}_{0,*}(Z)\oplus E^{\infty}_{0,*}(N\setminus Z) \arrow{r} & G_2 \arrow{r} & E^{\infty}_{1,*-1}(Z)\oplus E^{\infty}_{1,*-1}(N\setminus Z) \arrow{r} & 0
\end{tikzcd}
\]
Here, $H_1$ is the homology of $(\partial Z\times S^1)_1$ in $\partial Z\times S^1$, $G_1$ is the homology of $E_{(1)}$ in $H_*(E)$,
 and $G_2$ is the homology of $Z_{(1)}$ plus the homology of $(N\setminus Z)_{(1)}$ in $H_*(Z)\oplus H_*(N\setminus Z)$.
The kernel of $G_1\to G_2$ then is $K'$ (since $\tilde{K}\cap E^{\infty}(E)\subset E_{1,*}^{\infty}(E)$), and can be identified with $f(K')\subset \ker(g)$. Now this has a unique preimage
in $E^{\infty}_{1,*-1}(\partial Z\times S^1)$ by \cref{kerdelinfty}.
Thus (and also because the top left vertical map is an isomorphism), to $x$ in $K'$ there is a unique preimage $\tilde{x}$ in $H_1$.\\
Now look at
\[
\begin{tikzcd}
	0 \arrow{r} & E^{\infty}_{0,*}(\partial Z\times S^1) \arrow{r} \arrow{d} & H_1 \arrow{r}{f} \arrow{d} & E^{\infty}_{1,*-1}(\partial Z\times S^1) \arrow{r} \arrow{d}{i_1} & 0\\
	0 \arrow{r} & E^{\infty}_{0,*}(\partial Z)\oplus E^{\infty}_{0,*}(X_1^{n-1}\setminus \partial Z) \arrow{r} \arrow{d}{j_1} & H_2 \arrow{r} \arrow{d}{j_2} &
	E^{\infty}_{1,*-1}(\partial Z)\oplus E^{\infty}_{1,*-1}(X_1^{n-1}\setminus \partial Z) \arrow{r} \arrow{d}{i_2} & 0\\
	0 \arrow{r} & E^{\infty}_{0,*}(Z)\oplus E^{\infty}_{0,*}(N\setminus Z) \arrow{r} & G_2 \arrow{r} & E^{\infty}_{1,*-1}(Z)\oplus E^{\infty}_{1,*-1}(N\setminus Z) \arrow{r} & 0
\end{tikzcd}
\]
We want to show that $\tilde{x}$ is in the kernel of $H_1\to H_2$.
We have that $i_1(f(\tilde{x}))=0$ by the choice of $\tilde{x}$. If $\tilde{x}$ was not $0$ in $H_2$, it would be in the non-trivial image of $j_2$. But this would mean
that $\tilde{x}$ in $G_2$ would not be $0$ ($j_1$ is injective because $h^{\infty}$ is), a contradiction.
\end{proof}
With this preliminary work, we can look at
\[
\begin{tikzcd}
	\hdots \arrow{r} & H_j(X_1^{n-1}\setminus \partial Z) \oplus H_j(\partial Z)\arrow{r} \arrow{d} & H_j(X_1^{n-1}) \arrow{r}{\partial_j} \arrow{d} &
	H_{j-1}(\partial Z\times S^1) \arrow["g_1",r] \arrow{d} & \hdots\\
	\hdots \arrow{r} & H_j(N\setminus Z)\oplus H_j(Z) \arrow{r} & H_j(N) \arrow{r}{\partial_j} & H_{j-1}(E) \arrow["g_2",r] & \hdots
\end{tikzcd}
\]
where we write $\partial Z$ for $\partial Z\cap X_1^{n-1}$. The image of the bottom left map equals
\[
(H_j(N\setminus Z)\oplus H_j(Z))/\text{im}[H_j(E)\to H_j(N\setminus Z)\oplus H_j(Z)].
\]
By the previous discussion, this is surjected on by the image of the left vertical map.\\
By \cref{kerdel}, the kernel of $g_2$ is in the image of the kernel of $g_1$ under the right vertical map.
All in all, this shows that $H_j(X_1^{n-1})\to H_j(N)$ is surjective by the same diagram chase one does to prove the four lemma about surjectivity. We are done.\\
\begin{remark}\label{surjrem}
	We can construct manifolds where a connected component of $Z^*$ is indeed not an integer homology disk. There is an embedding $\SO(3)\to S^5$
	with trivial normal bundle (this factors through $\SO(3)\to S^2\times S^2\to S^5$, where the first embedding comes from
	sending the first two columns to the respective points in the $S^2$'s), so there is an embedding
	$D_{\Q}^3\to D^5$, where $D_{\Q}^3:=\SO(3)\setminus \mathring{D}^3$, such that the boundary $S^1$ is mapped to the boundary of the standard $D^3\subset D^5$.\\
	Take some (quasi)toric 8-manifold and restrict the
	$T^4$-action to a $T^3$-action in general position. The orbit space of a closed neighborhood $U$ of a non-free $T^3$-orbit is then
	$D^5$ with a prescribed embedding $S^2\to \partial D^5$, and we may extend this to the embedding $D_{\Q}^3\to D^5$.
	It is clear that we can extend the action from $\partial U$ to $U$ such that the isotropy submanifold is precisely $D_{\Q}^3\times T^3$.
	The resulting manifold is equivariantly formal over $\Z$ if and only if the order of the isotropy group is odd
	(we will see the other direction in \cref{necEF}).
\end{remark}

\newpage

\section{A necessary condition for equivariant formality}\label{necEF}
Here we want to formulate and prove the following theorem, where the coefficients are taken to be $\Z$.
We use the notation in \cref{suffEF}.
\begin{theorem}\label{theonecEF}
	 Consider an action in general position of $T=T^{n-1}$ on the compact manifold $M$ of dimension $6$ or $2n\geq 10$ such that
	 \begin{itemize}
	 	\item The odd (co)homology of $M$ vanishes.
	 	\item Whenever $x$ is a point whose stabilizer is finite, there is a neighborhood of $x$ in which this is the only non-trivial stabilizer subgroup.
	 \end{itemize}
	 The following statements hold.
	\begin{enumerate}
		\item For all closed subgroups $H\subset T$ (not necessarily connected), every connected component of $M^H$ intersects $M^T$ non-trivially.
		\item Any connected component $Q^*$ of $Z^*$ is admissible.
		\item $M^*_k$ is $k-1$-acyclic for $k\leq n-2$.
		\item $M^*$ is a homology sphere.
	\end{enumerate}
\end{theorem}
\begin{remark}
	We excluded the case $n=4$ due to technical difficulties in the proof for the last point (the first three items can be proven just as described
	in the proof of \cref{theonecEF}).\\
	The author would be thankful for hints or remarks, even entire solutions, to how this can be solved.
\end{remark}
\begin{proof}
	We may assume that $Z\neq \emptyset$, because this is already treated in \cite{AM23}.
\begin{enumerate}
\item
This is standard if $H$ has positive dimension. If it doesn't, then any connected component $Q$ of $M^H$ is contained in the interior of $N$,
and for any prime $p$ dividing the order of the stabilizer group of $Q$ we have the estimate
\[
\dim(H^{odd}(Q;\Z_p))\leq \dim(H^{odd}(M;\Z_p))=0,
\]
so in particular $H^{odd}(Q;\Q)=0$, which is impossible due to $\chi(Q)=0$.\\

\item In this item, we take (co)homology with $\Z_p$-coefficients unless stated otherwise.\\
It suffices to show the statement for any prime $p$ that divides the order of the kernel of the $T$-action on $Q$. Then
\[
\dim H^{odd}(Q)\leq \dim H^{odd}(M)=0,
\]
so $Q$ is equivariantly formal over $\Z_p$. After dividing out the kernel, the $T^{n-1}$-action on $Q$ is locally standard.
Using \cref{thm:ABFP} for $\Z_p$-coefficients on $Q$, we obtain that
\[
H_T^{*+n-1}(Q,Q_{n-2})=H_T^{*+n-1}((N\cap Q)^*,(X\cap Q)^*)=H^{*+n-1}((N\cap Q)^*,(X\cap Q)^*)=0
\]
for degree $*<0$, so the reduced $\Z_p$-cohomology of $Q^*$ is at most non-trivial in degree $n-2$, and $H^{n-2}(Q^*,\partial Q^*)=0$. By Lefschetz duality,
$H_1(Q^*)=0$ and thus $H_1(Q^*,\partial Q^*)=0$. We finally conclude $H^{n-2}(Q^*)=0$ by using Lefschetz duality again.
It follows that Tor$(H_{\geq 1}(Q^*;\Z),\Z_p)=0$, so multiplication with $p$ in $H_*(Q^*;\Z)$ is an isomorphism in positive degrees.\\

\item
This is very similar to the arguments in \cite{AM23}. We only need to show that $M_{n-2}^*$ is $n-3$-acyclic,
since $H_{\leq k-1}(M_k^*)\to H_{\leq k-1}(M_{k+1}^*)$ is an isomorphism in positive degrees for $k\leq n-3$.\\
By \cref{thm:ABFP} and \cref{rem:ABFP}, we have a long exact sequence
\[
0\to H_T^*(M)\to H_T^*(M_0)\to H_T^{*+1}(M_1,M_0)\to \hdots \to H_T^{*+n-1}(M,M_{n-2})\to 0.
\]
Note that, although the action does not have connected stabilizers,
\[
H^{\leq i}_T(M_i,M_{i-1})=\bigoplus\limits_{Q:\dim(Q)=i} H^{\leq i}_T(Q_i,\partial Q_i)=\bigoplus\limits_{Q:\dim(Q)=i}  H^{\leq i}(Q_i^*,\partial Q_i^*),
\]
and every summand is non-trivial only in degree $i$. Hence, the above sequence for $*=0$ becomes
\[
0\to \Z\to \bigoplus\limits_{Q:\dim Q=0} \Z \overset{\delta_0}{\to} \hdots \overset{\delta_{n-3}}{\to} \bigoplus\limits_{Q:\dim Q=n-2} \Z \to H_T^{n-1}(M,M_{n-2})\to 0,
\]
(for $*<0$ we obtain $H_T^{\leq n-2}(M,M_{n-2})=0$) and this is still the cohomological cochain complex up to degree $n-3$ for the cell complex defined by $Q_{n-2}$. The acyclicity of this sequence implies that indeed
$M_{n-2}^*$ is $n-3$-acyclic.\\

\item Now we know that $X$ respectively $X^*$ has the properties as in \cref{homologyX*} and the remaining section.
Let us now turn to the statement about $M^*$. Since $M^*$ is certainly a rational homology sphere, and certainly $H_1(M^*)=H_1(M)=0$,
we can already deduce that $M^*$ is an integer homology sphere for $2n=6$. So let $n\geq 5$.\\

When we choose the degree to be less than $0$ in the ABFP-sequence, we obtain $0=H_T^{k}(M,M_{n-2})=H_T^{k}(N,X)$ (excision)
for $k\leq n-2$, which means that $H_T^k(N)\to H_T^k(X)$ is an isomorphism for $k\leq n-3$ and injective for $k=n-2$.
We want to argue that the same statement holds for ordinary cohomology (which would immediately follow from the action being free on $N$).\\

In fact, we will show at first that $H^*(N^*)=0$ for all degrees except $0$ and $2$, and then deal with degree $2$. Since we have $\pi_1(N^*)\cong \pi_1(M)$, we already know $H_1(N^*)=H^{1}(N^*)=0$,
and so $H^{n-1}(N^*)=0$ by Lefschetz duality. This means that we only need to worry about degrees $3\leq * \leq n-2$ in that regard.\\

This is void for $n=4$, so assume $n\geq 5$ and consider the diagram
\[
\begin{tikzcd}
	\hdots \arrow{r} & H^{*-1}(E^*) \arrow{r} \arrow{d} & H_T^*(N) \arrow{r} \arrow["g_1",d] & H^{*}((N\setminus Z)^*)\oplus H_T^{*}(Z) \arrow{d} \arrow{r} & \hdots\\
	\hdots \arrow{r} & H^{*-1}(\partial Z^*\times S^1) \arrow{r} & H_T^*(X) \arrow["g_2",r] & H^{*}((X\setminus \partial Z)^*)\oplus H_T^{*}(\partial Z) \arrow{r} & \hdots
\end{tikzcd}
\]
We need to understand $H_T^*(Z)\to H_T^*(\partial Z)$. We may restrict to a connected component of $Z$ (also called $Z$), and argue there.\\
The $T$-action on $Z$ has a finite stabilizer $\Z_p\subset T$, which is contained in a subcircle $S\subset T$. We may assume that $T=(S^1)^{n-1}$
and $S=S^1\times \{e\}\times \hdots \times \{e\}$, so that $T'=\{e\}\times T^{n-2}$ acts freely on $Z$. We have
\[
H_T^*(Z)=H^*(Z\times_T (S^{\infty})^{n-1})=H^*(Z/T' \times_{T/T'} S^{\infty})=H^*(Z/T' \times_{S^1/Z_p} L_p^{\infty}),
\]
so $H_T^*(Z)$ is the cohomology of the total space of an $L_p^{\infty}$-bundle over $Z^*$ (and similar for $\partial Z^*$).
Remember that $H^*(L_p^{\infty})$ is $\Z_p$ in even positive degrees and $0$ for odd degrees. It follows that
$E_2^{p,q}(Z)=H^p(Z^*; H^q(L_p^{\infty}))=0$ except for $q=0$ or $p=0$, because multiplication with $p$ is an isomorphism in $H_{\geq 1}(Z^*)$.
Hence, we have $H_T^*(Z)=H^*(Z)\oplus H^*(L_p^{\infty})$, and the kernel of $H_T^*(Z)\to H_T^*(\partial Z)$ is precisely $H^*(Z^*)$ for $*\geq 1$.\\

Now, we have $H^*(X^*)=H^*((X\setminus \partial Z)^*)=H^*(\partial Z^*)=0$ for $3\leq *\leq n-3$, so by the bottom sequence of the diagram we get that
$H^*_T(X)\to H_T^*(\partial Z)$, and thus $H^*_T(N)\to H_T^*(\partial Z)$, is an isomorphism, which implies that
\[
H_T^*(N) \to H^{*}((N\setminus Z)^*)\oplus H_T^{*}(Z)
\]
is injective and hits $H^{*}((N\setminus Z)^*)\oplus H^{*}(Z)$ trivially. The last statement also holds for $*=n-2$, since then $H^*((N\setminus Z)^*)\to H^*((X\setminus \partial Z)^*)$
is the $0$-map as well (the first group is torsion), and both $g_1$ and $g_2$ are injective. Hence
\[
H^{*}((N\setminus Z)^*)\oplus H^*(Z)\to H^{*}(E^*)
\]
is injective, and $H^{*-1}((N\setminus Z)^*)\oplus H_T^{*-1}(Z)\to H^{*-1}(E)$ and thus
$$H^{*-1}((N\setminus Z)^*)\oplus H^{*-1}(Z)\to H^{*-1}(E)$$ are surjective. This shows $H^*(N^*)=0$ for $3\leq *\leq n-2$.\\

Let us now consider $H^2(N^*)\to H^2(X^*)$, still assuming $n\geq 5$. We know that both groups are torsion free, and
we know that $H^2(N^*;\Q)\to H^2(X^*;\Q)$ is an isomorphism (since $H^2_T(N^*;\Q)=H^2(N^*;\Q)$ and similar for $X$).
So it is immediate that both $H^2(N^*)\to H^2(X^*)$ and $H_2(X^*)\to H_2(N^*)$ are injective. By the MVs belonging to $M^*=M_{n-2}\cup N^*$,
we see that $H_*(M^*)=H_*(S^{n+1})$ except in degree $2$, where there could be torsion. But this can be ruled out using Poincaré duality.
\end{enumerate}
\end{proof}

\newpage

\section{Appendix}
\subsection{Serre spectral sequence}
Here we fix some notation regarding the Serre spectral sequence for fiber bundles (over $\Z$ or $\Q$). We should note that we always talk about singular homology with coefficients
either in $\Q$ or $\Z$ whenever we talk about 'homology'. As long as not specified, coefficients are allowed to be both $\Q$ or $\Z$. See e.g. \cite{Hat04} for details.\\
Let $F\to M\to B$ be a fiber bundle over a connected CW-complex $B$, and assume that the homology of $B$ or $F$ over $\Z$ respectively $\Q$ is finitely generated. The point of the Serre spectral sequence is to compute the 
homology of the space $M$ from the homology of $B$ and $F$ via a first quadrant spectral sequence, only subject to the small condition that
the monodromy representation $\pi_1(B)\to \Aut(H_*(F))$ is trivial. This is true, for example, when the bundle restricted to the one-skeleton
of $B$ is trivial, so in particular for free group actions of a connected Lie group. Usually, one starts from the second page, which has the form
\[
E^2_{p,q}=H_p(M^*,H_q(F))=H_p(M^*)\otimes H_q(F)
\]
(where the last equation always holds over $\Q$, and also over $\Z$ if the homology of $F$, for example, is torsion free, which we assume from now on)
and the differential $d^2_{p,q}\colon E^2_{p,q}\to E^2_{p-2,q+1}$ (it will not matter for us how this differential looks like). Now, the third page is the homology of
the second page, that is,
\[
E^3_{p,q}:=\frac{\ker(d^2_{p,q}\colon E^2_{p,q}\to E^2_{p-2,q+1})}{\text{im}(d^2_{p+2,q-1}\colon E^2_{p+2,q-1}\to E^2_{p,q})}.
\]
The new differential now is of the form $d^3_{p,q}\colon E^2_{p,q}\to E^2_{p-3,q+2}$. Again, we do not have to know how this looks like exactly.\\
In general, we define $E^{r+1}_{p,q}$ from $E^r_{p,q}$ by
\[
E^{r+1}_{p,q}:=\frac{\ker(d^r_{p,q}\colon E^r_{p,q}\to E^r_{p-r,q+r-1})}{\text{im}(d^r_{p+r,q-r+1}\colon E^r_{p+r,q-r+1}\to E^r_{p,q})}
\]
with differential $d^{r+1}_{p,q}\colon E^{r+1}_{p,q}\to E^{r+1}_{p-r,q+r-1}$. Note that, for $r$ big enough, $E^r=E^{r+1}=E^{r+2}=\hdots$,
because we assumed the homologies of $B$ or $F$ to be finitely generated.
We then define $E^{\infty}_{p,q}:=E^r_{p,q}$ and say that the spectral sequence $(E^r,d^r)$ converges against $E^{\infty}$.
The homology of $M$ is encoded in $E^{\infty}$.
Denote by $(B)_p$ the $p$-skeleton of $B$, and by $M_p$ its preimage under $\pi\colon M\to B$.
\begin{theorem}
	We have
	\[
	E^{\infty}_{p,q}=\text{im}(H_{p+q}(M_p)\to H_{p+q}(M))/\text{im}(H_{p+q}(M_{p-1})\to H_{p+q}(M))
	\]
	and therefore, over $\Q$, the isomorphism
	\[
	H_k(M)\cong \bigoplus\limits_{p+q=k} E^{\infty}_{p,q}.
	\]
\end{theorem}
\begin{remark}\label{rem:groupext}
	Over $\Z$, there is no reason to assume that
	\[
	H_k(M)\cong \bigoplus\limits_{p+q=k} E^{\infty}_{p,q}
	\]
	since short exact sequences of abelian groups do not necessarily split. Instead, in order to calculate $H_k(M)$ for a given $k$, we get a series of extension problems to solve.
	That is, assuming we know the image $K$ of $H_k(M_{p-1})\to H_k(M)$ (which is $E^{\infty}_{0,k}$ for $p=1$), the image $H$ of $H_k(M_{p})\to H_k(M)$ sits in the short exact sequence
	\[
	0\to K\to H\to E^{\infty}_{p,k-p}\to 0.
	\]
\end{remark}
\begin{lemma}\label{lem:naturalityspectralsequence}
	Let $F_1\to X\to B_1$ and $F_2\to M\to B_2$ be bundles as above and consider a bundle map $f\colon X\to M$ covering $g\colon B_1\to B_2$.
	We consider the spectral sequences $E^r(X)$ and $E^r(M)$ belonging to $X$ and $M$, respectively.
	Then the following statements hold:
	\begin{enumerate}
		\item There is a map $f_*^r\colon E^r(X)\to E^r(M)$ which commutes with the differentials, where $f_*^{r+1}$ is induced by $f_*^r$ in
		the canonical way and $f_*^2$ is the canonical map $g_*\otimes i_*$, where $i_*\colon H_*(F_1)\to H_*(F_2)$ is the well-defined homomorphism
		in homology induced by $f$ and generic fiber inclusions $i_1$ and $i_2$. 
		\item The map $H_*(X)\to H_*(M)$ induced by $f$ is compatible with $f_*^{\infty}$ in the sense that it induces maps between the group extensions from above.
	\end{enumerate}
\end{lemma}
Everything we just explained for homology works in the same way for cohomology, the difference being that the differentials
go to the reverse direction. Of course, one might expect that the ring structure on the cohomology does survive
when going to the next page, and is in some sense compatible with the differentials. While this is true, it is not needed
in that paper and so we omit this.\\

Now let a connected Lie group $G$ act freely on a CW-complex $M$. We get a fiber bundle $G\to M\to M/G$ which is trivial over $(M/G)_1$.
So we can always use the Serre spectral sequence in these situations, even with $\Z$-coefficients.
\begin{remark}\label{serreQ}
	It is a crucial fact that, if $G=T^n$ is abelian and
	acts almost freely (that is, only with dicrete isotropies), then the Serre spectral sequence even works in this case (together
	with all naturality properties), under the restriction that one takes $\Q$-coefficients. We will outline the reason here
    (this is due to \cite{Z16}, and persosnal correspondence). We consider the homotopy fiber $F_p$ of $p$ in $M\to M_{T^n} \overset{p}{\to} BT^n$, and the following bundle
	\[
	\Omega BT^n \to F_p \to M_{T^n}.
	\]
	Together with the natural homotopy equivalence $M\to F_p$, there can be constructed a homotopy equivalence $h\colon T^n\to \Omega BT^n$
	such that the diagram
	\[
	\begin{tikzcd}
		T^n \arrow{r} \arrow{d} & M \arrow{r} \arrow{d} & M_{T^n} \\
		\Omega BT^n \arrow{r} & F_p \arrow{ru}
	\end{tikzcd}
	\]
	commutes. That is, we can say that $T^n\to M\to M_{T^n}$ is a 'fibration up to homotopy equivalence'. Now, $H^*(M/T;\Q)\to H_T^*(M;\Q)=H^*(M_{T^n})$
	is an isomorphism, and the monodromy of $\Omega BT^n \to F_p \to M_{T^n}$ is trivial, so its Serre spectral sequence computes the (co)homology of $M$ with rational coefficients.
	The naturality conditions now come from the naturality of the constructed fibration $\Omega BT^n \to F_p \to M_{T^n}$.
\end{remark}

\newpage

\bibliographystyle{amsalpha}

\end{document}